\documentclass[12pt,a4paper]{amsart}

\usepackage{amssymb,amsmath,amsthm,amstext,amscd}
\usepackage[english]{babel}

\title[Accelerated spatial approximations]{Accelerated spatial approximations for time discretized stochastic partial differential equations}

\author[E.J. Hall]{Eric Joseph Hall}
\address{School of Mathematics\\University of Edinburgh\\King's Buildings\\Edinburgh, EH9 3JZ\\UK}
\email{e.hall@ed.ac.uk}


\numberwithin{equation}{section}
\newtheorem{thm}{Theorem}[section]          
\newtheorem{cor}[thm]{Corollary}             
\newtheorem{lem}[thm]{Lemma}                 
\newtheorem{asm}[thm]{Assumption}            
\theoremstyle{definition}
\theoremstyle{definition}
\newtheorem{rmk}[thm]{Remark}                

\begin{document}
\begin{abstract}
The present article investigates the convergence of a class of space-time discretization schemes for the Cauchy problem for linear parabolic stochastic partial differential equations (SPDEs) defined on the whole space. Sufficient conditions are given for accelerating the convergence of the scheme with respect to the spatial approximation to higher order accuracy by an application of Richardson's method. This work extends the results of Gy{\"o}ngy and Krylov [\emph{SIAM J.\ Math.\ Anal.}, 42 (2010), pp. 2275--2296] to schemes that discretize in time as well as space.
\end{abstract}

\maketitle

\section{Introduction}\label{sec: Introduction}
For a fixed $\tau \in (0,1)$, we consider the equation 
\begin{equation}
\label{eqn: implicit space-time scheme}
v^{h}_{i} = v^{h}_{i-1} + \left( L^{h}_{i} v^{h}_{i} + f_{i} \right) \tau  + \sum_{\rho=1}^{d_{1}} \left( M^{h,\rho}_{i-1} v^{h}_{i-1} + g^{\rho}_{i-1} \right) \xi^{\rho}_{i}
\end{equation}
for $i \in \{1, \dots, n\}$ and $(\omega, x) \in \Omega \times G_{h}$ with a given initial condition, where $G_{h}$ is the space grid 
$$G_{h} := \{ \lambda_{1} h + \dots + \lambda_{p} h ; \lambda_{1}, \dots, \lambda_{p} \in \Lambda \cup (-\Lambda)\}$$ 
with mesh size $h \in \mathbf{R}\setminus\{0\}$ for a finite subset $\Lambda \subset \mathbf{R}^{d}$, for integer $d \geq 1$, containing the origin. For a fixed $T \in (0, \infty)$ we define the time grid 
$$T_{\tau} := \{ t_{i} = i\tau ; i \in \{0, 1, \dots, n\}, \tau n = T\},$$ 
partitioning $[0,T]$ with mesh size $\tau$, and note that $v^{h} = v^{h}(\omega, t, x)$ depends on the parameter $\tau$ as well as $h$, since we have used the convention of writing $v^{h}_{i}$ in place of $v^{h} (t_{i})$ for $t_{i} \in T_{\tau}$. In particular, let $\xi^{\rho}_{i} = \Delta w^{\rho}(t_{i-1}) := w^{\rho}(t_{i}) - w^{\rho}(t_{i-1})$ be the $i$th increment of $w^{\rho}$ with respect to $T_{\tau}$, where, for integer $d_{1} \geq 1$, $(w^{\rho})_{\rho=1}^{d_{1}}$ is a given sequence of independent Wiener processes carried by the stochastic basis $(\Omega, \mathcal{F}, \mathcal{F}(t), P)$ that is complete with respect the filtration $\mathcal{F}(t)$ for $t \in [0,T]$. For each $i \in \{0, \dots, n\}$, the $L^{h}_{i}$ and $M^{h,\rho}_{i}$ are difference operators given by $L^{h}_{i} \phi := \mathfrak{a}^{\lambda \mu}_{i} \delta_{h,\lambda} \delta_{-h,\mu} \phi$ and $M^{h,\rho}_{i} \phi := \mathfrak{b}^{\lambda \rho}_{i} \delta_{h,\lambda} \phi$, for $\rho \in \{1, \dots, d_{1}\}$, where repeated indices indicate summation over $\lambda, \mu \in \Lambda$. We assume that $\mathfrak{a}^{\lambda\mu}_{i} =\mathfrak{a}^{\lambda \mu}_{i}(x)$ and $\mathfrak{b}^{\lambda}_{i} = (\mathfrak{b}^{\lambda\rho}_{i}(x))_{\rho=1}^{d_{1}}$ are real-valued $\mathcal{P}\times \mathcal{B}$-measurable functions on $\Omega \times T_{\tau} \times \mathbf{R}^{d}$ for all $\lambda, \mu \in \Lambda$ and further that $\mathfrak{a}^{\lambda\mu}_{i} = \mathfrak{a}^{\mu\lambda}_{i}$. Here $\mathcal{P}$ denotes the $\sigma$-algebra of predictable subsets of $\Omega \times [0,\infty)$ generated by $\mathcal{F}(t)$ and $\mathcal{B} = \mathcal{B}(\mathbf{R}^{d})$ denotes the $\sigma$-algebra of Borel subsets of $\mathbf{R}^{d}$. The spatial differences above are defined by $$\delta_{h,\lambda} \phi(x) := \frac{\phi(x+h\lambda) - \phi(x)}{h}$$ for $\lambda \in \mathbf{R}^{d}\setminus \{0\}$ and by the identity for $\lambda = 0$. We note that from this definition one can obtain both the so called ``forward'' and ``backward'' differences as $h$ can be positive or negative. 

Together with \eqref{eqn: implicit space-time scheme} we consider 
\begin{equation}
\label{eqn: implicit time scheme}
v_{i} = v_{i-1} + \left( \mathcal{L}_{i} v_{i} + f_{i} \right) \tau  + \sum_{\rho=1}^{d_{1}} \left( \mathcal{M}^{\rho}_{i-1} v_{i-1} + g^{\rho}_{i-1} \right) \xi^{\rho}_{i}
\end{equation}
for $i \in \{1, \dots, n\}$ and $(\omega,x) \in \Omega \times \mathbf{R}^{d}$ with a given initial condition. Here  $\mathcal{L}_{i} = \mathcal{L}(t_{i})$ and $\mathcal{M}^{\rho}_{i} = \mathcal{M}^{\rho}(t_{i})$ are second order and first order differential operators given by $\mathcal{L}(t) := a^{\alpha\beta}(t) D_{\alpha}D_{\beta}$ and $\mathcal{M}^{\rho}(t) := b^{\alpha\rho}(t) D_{\alpha}$, respectively, where the summation is over $\alpha, \beta \in \{0, 1,\dots, d\}$ and where $D_{\alpha} = \partial / \partial x^{\alpha}$, for $\alpha \in \{1, \dots, d\}$, while $D_{0}$ is the identity. For each $\alpha$ and $\beta$ we assume that $a^{\alpha\beta}(t) = a^{\alpha\beta}(t,x)$ and $b^{\alpha}(t) = (b^{\alpha\rho}(t,x) )_{\rho=1}^{d_{1}}$ are real-valued $\mathcal{P} \times \mathcal{B}$-measurable functions on $\Omega \times [0,T] \times \mathbf{R}^{d}$, and further that $a^{\alpha\beta}(t) = a^{\beta\alpha}(t)$ for all $t \in [0,T]$.

Equations \eqref{eqn: implicit space-time scheme} and \eqref{eqn: implicit time scheme} represent discrete schemes for approximating the solution to the Cauchy problem for 
\begin{equation}
\label{eqn: spde}
du(t,x) = (\mathcal{L}u(t,x) + f(t,x))dt + \sum_{\rho=1}^{d_{1}} (\mathcal{M}^{\rho}u(t,x) + g^{\rho}(t,x))dw^{\rho}(t)
\end{equation}
for $(\omega,t,x) \in \Omega \times [0,T] \times \mathbf{R}^{d}$ with a given initial condition $u_{0}(x) = u(0,x)$. 
Under certain compatibility assumptions, equation \eqref{eqn: implicit space-time scheme} represents an implicit space-time scheme for approximating the solution to the Cauchy problem for \eqref{eqn: spde} by replacing the differential operators with finite differences and by carrying out an implicit Euler method in time. In a similar fashion, \eqref{eqn: implicit time scheme} represents an implicit Euler method for approximating the solution to the Cauchy problem for \eqref{eqn: spde} in time. Second order linear parabolic SPDE such as \eqref{eqn: spde} arise in the nonlinear filtering of partially observable diffusion processes as the Zakai equation (\cite{Kunita:1981,Pardoux:1979,Zakai:1969,BainCrisan:2009}). Since analytic solutions to \eqref{eqn: spde} are difficult to obtain, there is a keen interest in providing accurate numerical schemes for its solution. 

Our aim is to show that the strong convergence of the spatial discretization for the space-time scheme \eqref{eqn: implicit space-time scheme} to the solution of the Cauchy problem for \eqref{eqn: spde} can be accelerated to any order of accuracy with respect to the computational effort. In general, the error of finite difference approximations in the space variable for such equations is proportional to the mesh size $h$, for example, see \cite{Yoo:1998th,Yoo:2000}. We show the strong convergence of the solution of the space-time scheme to the solution of the time scheme \eqref{eqn: implicit time scheme} can be accelerated to higher order accuracy by taking suitable mixtures of approximations using different mesh sizes.

This technique for obtaining higher order convergence, often referred to as \emph{Richardson's method} after L.F.\ Richardson who used the idea to accelerate the convergence of finite difference schemes to deterministic partial differential equations (PDE) (see \cite{Richardson:1911,RichardsonGaunt:1927}), falls under a broadly applicable category of extrapolation techniques, for instance see the survey articles \cite{Brezinski:2000,Joyce:1971}. In particular, in \cite{TalayTubaro:1990,MalliavinThalmaier:2003,KloedenPlatenHofmann:1995} Richardson's method is implemented to accelerate the weak convergence of Euler approximations for stochastic differential equations. Recently, in \cite{GyongyKrylov:2010} Gy{\"o}ngy and Krylov considered a semi-discrete scheme for solving \eqref{eqn: spde} which discretized via finite differences in the space variable, while allowing the scheme to vary continuously in time, and showed that the strong convergence of the spatial approximation can be accelerated by Richardson's method. The current paper extends these results to the implicit space-time scheme \eqref{eqn: implicit space-time scheme}. 

We must mention that for the present scheme one cannot also accelerate in time unless certain commutators of the differential operator $\mathcal{M}^{\rho}$ in equation \eqref{eqn: spde} vanish, see \cite{DavieGaines:2000}. For \emph{deterministic} PDE we plan to address the simultaneous acceleration of the convergence of approximations with respect to space and time in a future paper. Results concerning acceleration for monotone finite difference schemes for degenerate parabolic and elliptic PDE are given in \cite{GyongyKrylov:2011}, however our scheme is not necessarily monotone. 

In the next section, we present our assumptions as well as some preliminaries. Then in Section \ref{sec: Main Results} we record the main results, namely Theorems \ref{thm: expansion and estimate for the expansion error}, \ref{thm: expansion and estimate for differences of expansion error}, and \ref{thm: acceleration}, the last of which says that the convergence of the spatial approximation can be accelerated to any order of accuracy. In Section \ref{sec: Auxiliary Results} we provide results which will be needed for the proofs of Theorems \ref{thm: expansion and estimate for the expansion error} and \ref{thm: expansion and estimate for differences of expansion error}. In particular, we recall the solvability of the space-time scheme \eqref{eqn: implicit space-time scheme}, for the convenience of the reader, and present a new contribution---an estimate for the supremum of the solution to the scheme in appropriate spaces that is independent of $h$, the spatial mesh size. In Section \ref{sec: Proof of Main Results} we give the proof of a more general result and show that it implies Theorem \ref{thm: expansion and estimate for differences of expansion error} and hence Theorem \ref{thm: expansion and estimate for the expansion error}. 

We end with some notation that will be used throughout this work. Let $\ell^{2}(G_{h})$ be the set of real-valued functions $\phi$ on $G_{h}$ such that $$| \phi |_{l^{2}(G_{h})}^{2} := |h|^{d} \sum_{x\in G_{h}} |\phi (x) |^{2} < \infty$$ and note that this notation will also be used for functions in $\ell^{2}(\mathbf{R}^{d})$. 

For a nonnegative integer $m$, let $W_{2}^{m} = W_{2}^{m}(\mathbf{R}^{d})$ be the usual Hilbert-Sobolev space of functions on $\mathbf{R}^{d}$ with norm $\| \cdot \|_{m}$. We note that for $L^{2} = L^{2}(\mathbf{R}^{d}) = W^{0}_{2}$ the norm will be denoted by $\| \cdot \|_{0}$. We use the notation $D^{l}\phi$ for the collection of all $l$th order spatial derivatives of $\phi$. Let $$\mathbf{W}^{m}_{2}(T) := L^{2}(\Omega\times [0,T], \mathcal{P}, W_{2}^{m})$$ denote the space of $W_{2}^{m}$-valued square integrable predictable processes on $\Omega \times [0,T]$. These are the natural spaces in which to seek solutions to \eqref{eqn: spde}.

\section{Preliminaries and Assumptions}\label{sec: Preliminaries and Assumptions}

We begin by setting some assumptions on our operators and recalling well known results concerning the solvability and rates of convergence for our schemes. In particular, we will discuss an $\ell^{2}(G_{h})$ notion of solution and an $L^{2}$ notion of solution and recall an important lemma relating these function spaces. 

An $L^{2}$-valued continuous process $u = (u(t))_{t \in [0,T]}$ is called a generalized solution to \eqref{eqn: spde} if $u \in W^{1}_{2}$ for almost every $(\omega, t) \in \Omega \times [0,T]$, $$\int_{0}^{T} \| u(t) \|_{1}^{2} \, dt < \infty$$ almost surely, and 
\begin{equation*}
\begin{split}
(u(t),\phi) =  \int_{0}^{t} ((a^{0\beta} - D_{\alpha}a^{\alpha\beta}) D_{\beta} u(s) + f(x), \phi) - (a^{\alpha\beta}D_{\beta}u , D_{\alpha} \phi) \,ds \\ +  (u_{0}, \phi) + \sum_{\rho=1}^{d_{1}} \int_{0}^{t} (\mathcal{M}^{\rho} u(s) + g^{\rho}(s), \phi ) \, dw^{\rho}(s)
\end{split}
\end{equation*}
holds for all $t \in [0,T]$ and $\phi \in C_{0}^{\infty}(\mathbf{R}^{d})$. 

\begin{asm}
\label{asm: boundedness of coefficients}
For each $(\omega, t) \in \Omega \times [0,T]$ the functions $a^{\alpha\beta}$ are $m$ times and the functions $b^{\alpha}$ are $m+1$ times continuously differentiable in $x$. Moreover there exist constants $K_{0}$, \dots, $K_{m+1}$ such that for $l \leq m$
$$|D^{l} a^{\alpha\beta} | \leq K_{l}$$ and for $l \leq m+1$ $$|D^{l} b^{\alpha} |_{\ell_{2}} \leq K_{l}$$ 
for all values of $\alpha, \beta \in \{0, \dots, d\}$ and $(\omega,t,x) \in \Omega \times [0,T] \times \mathbf{R}^{d}$.
\end{asm}

\begin{asm}
\label{asm: strong stochastic parabolicity}
There exists a positive constant $\kappa$ such that $$\sum_{\alpha,\beta=1}^{d} (2a^{\alpha\beta} - b^{\alpha\rho}b^{\beta\rho})z^{\alpha}z^{\beta} \geq \kappa |z|^{2}$$ for all $(\omega,t,x) \in \Omega \times [0,T] \times \mathbf{R}^{d}$, $z \in \mathbf{R}^{d}$, and $\rho \in \{1, \dots, d_{1}\}$. 
\end{asm}

\begin{asm}
\label{asm: initial conditions and free terms}
The initial condition $u_{0} \in L^{2}(\Omega, \mathcal{F}_{0}, W_{2}^{m+1})$, the space of $W_{2}^{m+1}$-valued square integrable $\mathcal{F}_{0}$-measurable functions on $\Omega$. The $f$ and $g^{\rho}$, for $\rho \in \{1, \dots, d_{1}\},$ are predictable processes on $\Omega \times [0,T]$ taking values in $W_{2}^{m}$ and $W_{2}^{m+1}$, respectively. Moreover
$$ E \int_{0}^{T} ( \| f(t) \|_{m}^{2} + \| g(t) \|_{m+1}^{2} ) \, dt + E \| u_{0} \|_{m+1}^{2} < \infty, $$ 
where $\| g(t) \|_{l}^{2} := \sum_{\rho = 1}^{d_{1}} \| g(t)^{\rho} \|_{l}^{2}$. 
\end{asm}

Under Assumptions \ref{asm: boundedness of coefficients}, \ref{asm: strong stochastic parabolicity}, and \ref{asm: initial conditions and free terms}, the existence of a unique solution $u \in \mathbf{W}^{m+2}_{2}(T)$ to \eqref{eqn: spde} is a classical result (see for example \cite{Pardoux:1975,KrylovRozovskii:1977} or Theorem 5.1 from \cite{Krylov:1999}).

\begin{rmk}
\label{rmk: on Sobolev's embedding}
We note that by Sobolev's embedding of $W_{2}^{m} \subset \mathcal{C}_{b}$, the space of bounded continuous functions, for $m > d/2$ we can find a  continuous function of $x$ which is equal to $u_{0}$ almost everywhere for almost all $\omega \in \Omega$. Likewise, for each $(\omega, t) \in \Omega \times [0,T]$ there exists continuous functions of $x$ which coincide with $f(t)$ and $g^{\rho}(t)$ for almost every $x \in \mathbf{R}^{d}$. Thus, if Assumption \ref{asm: initial conditions and free terms} holds with $m > d/2$ we assume that $u_{0}$, $f(t)$, and $g^{\rho}(t)$ are continuous in $x$ for all $t \in [0,T]$. 
\end{rmk}

For a nonnegative integer $\mathfrak{m}$, let $\bar{\mathfrak{m}} := \mathfrak{m} \vee 1$ and $\Lambda_{0}:= \Lambda \setminus \{0\}$. We place the following additional requirements on our space-time scheme.

\begin{asm}
\label{asm: boundedness of space-time scheme coefficients}
For all $\omega \in \Omega$, for $i \in \{0, \dots, n\}$, for $\lambda,\mu \in \Lambda_{0}$, and for $\nu \in \Lambda$: the $\mathfrak{a}^{\lambda \mu}$ are $\bar{\mathfrak{m}}$ times continuously differentiable in $x$; the $\mathfrak{a}^{0\nu}$ and $\mathfrak{a}^{\nu 0}$ are $\mathfrak{m}$ times continuously differentiable in $x$; and the $\mathfrak{b}^{\nu}$ are $\mathfrak{m}$ times continuously differentiable in $x$.  Moreover there exist constants $A_{0}$, \dots, $A_{\bar{\mathfrak{m}}}$ such that for $\lambda,\mu \in \Lambda_{0}$ and $j\leq \bar{\mathfrak{m}}$ we have $$|D^{j}\mathfrak{a}^{\lambda\mu}| \leq A_{j}$$ and for $\lambda \in \Lambda$ and $j \leq \mathfrak{m}$ we have 
$$|D^{j} \mathfrak{a}^{\lambda 0}| \leq A_{j},\quad |D^{j} \mathfrak{a}^{0 \lambda}| \leq A_{j},\, \text{ and }\,  |D^{j} \mathfrak{b}^{\lambda}| \leq A_{j}$$
for all $(\omega, x) \in \Omega \times \mathbf{R}^{d}$ for $i \in \{0, \dots, n\}$. 
\end{asm}

\begin{asm}
\label{asm: stochastic parabolicity for space-time scheme}
There exists a positive constant $\kappa$ such that 
$$\sum_{\lambda, \mu \in \Lambda_{0}} (2\mathfrak{a}^{\lambda \mu} - \mathfrak{b}^{\lambda \rho}\mathfrak{b}^{\mu \rho})z_{\lambda}z_{\mu} \geq \kappa \sum_{\lambda \in \Lambda_{0}} z_{\lambda}^{2}$$
for all $(\omega, x) \in \Omega \times \mathbf{R}^{d}$, $i \in \{0, \dots, n\}$, $\rho \in \{ 1, \dots, d_{1}\}$, and numbers $z_{\lambda}$, $\lambda \in \Lambda_{0}$.
\end{asm}

For \eqref{eqn: implicit space-time scheme} to be consistent with \eqref{eqn: spde} we also require the following. 

\begin{asm}
\label{asm: consistency}
For $i \in \{0, \dots, n\}$ 
$$\mathfrak{a}^{00}_{i} = a^{00}_{i},$$ 
$$\sum_{\lambda \in \Lambda_{0}} \mathfrak{a}^{\lambda 0}_{i} \lambda^{\alpha} + \sum_{\mu \in \Lambda_{0}} \mathfrak{a}^{0\mu}_{i}\mu^{\alpha} = a^{\alpha0}_{i} + a^{0\alpha}_{i},$$ 
$$\sum_{\lambda,\mu \in \Lambda_{0}} \mathfrak{a}^{\lambda\mu}_{i} \lambda^{\alpha}\mu^{\beta} = a^{\alpha\beta}_{i},$$ 
$$\mathfrak{b}^{0 \rho}_{i} = b^{0\rho}_{i},$$
and $$\sum_{\lambda \in \Lambda_{0}} \mathfrak{b}^{\lambda \rho}_{i} \lambda^{\alpha} = b^{\alpha\rho}_{i}$$
for all $\alpha,\beta \in \{1, \dots, d\}$ and $\rho \in \{1, \dots, d_{1}\}$. 
\end{asm}

\begin{rmk}
\label{rmk: on equivalence of assumptions}
If $\Lambda_{0}$ is a basis for $\mathbf{R}^{d}$ and Assumption \ref{asm: consistency} holds then Assumption \ref{asm: boundedness of coefficients} implies \ref{asm: boundedness of space-time scheme coefficients} and \ref{asm: strong stochastic parabolicity} implies \ref{asm: stochastic parabolicity for space-time scheme} with $\mathfrak{m} = m$.
\end{rmk}

A solution $v^{h} = (v^{h}_{i})_{i=1}^{n}$ to \eqref{eqn: implicit space-time scheme} with an $\ell^{2}(G_{h})$-valued $\mathcal{F}_{0}$-measurable initial condition $v^{h}_{0}$ is understood as a sequence of $\ell^{2}(G_{h})$-valued random variables satisfying \eqref{eqn: implicit space-time scheme} on the grid $G_{h}$. The following result is well known and we provide it for the sake of completeness.

\begin{thm}
\label{thm: l2 valued solution to the space-time scheme}
Let $f$ and $g^{\rho}$ be $\mathcal{F}_{i}$-adapted $\ell^{2}(G_{h})$-valued processes and let $v^{h}_{0}$ be an $\mathcal{F}_{0}$-measurable $\ell^{2}(G_{h})$-valued initial condition. If Assumption \ref{asm: boundedness of space-time scheme coefficients} holds then \eqref{eqn: implicit space-time scheme} admits a unique $\ell^{2}(G_{h})$-valued solution for sufficiently small $\tau$. 
\end{thm}

\begin{proof}
By Assumption \ref{asm: boundedness of space-time scheme coefficients}, for each $i \in \{1, \dots, n\}$, equation \eqref{eqn: implicit space-time scheme} is a recursion with bounded linear operators on $\ell^{2}(G_{h})$. In particular, for each $h$ the operator norm of $\tau L^{h}$ is smaller than a constant less than $1$ for sufficiently small $\tau$, independently of $\omega \in \Omega$. Hence $(I - \tau L^{h})$ is invertible in $\ell^{2}(G_{h})$ for sufficiently small $\tau$, by the invertibility of operators in a neighborhood of the (invertible) identity operator $I$. Therefore, for $i \in \{1, \dots, n\}$ we are guaranteed an $\ell^{2}(G_{h})$-valued $\phi$ satisfying $(I - \tau L^{h}_{i})\phi = \psi$ for all $\psi \in \ell^{2}(G_{h})$ and moreover this solution is easily seen to be unique. Thus we can construct a unique solution to the scheme iteratively. 
\end{proof}

The rate of convergence of the solution $v^{h}$ of \eqref{eqn: implicit space-time scheme} (and $v$ of \eqref{eqn: implicit time scheme}) to the solution $u$ of \eqref{eqn: spde} with initial condition $u_{0}$ is known. In \cite{GyongyMillet:2005,GyongyMillet:2007,GyongyMillet:2009rc}, Gy{\"o}ngy and Millet obtained the rate of convergence for a class of equations in the nonlinear setting of which our schemes are a special case. Namely, in the situation of Remark \ref{rmk: on equivalence of assumptions}, if Assumptions \ref{asm: boundedness of coefficients}, \ref{asm: strong stochastic parabolicity}, and \ref{asm: initial conditions and free terms} hold with $a^{\alpha\beta}$, $b^{\alpha}$, $f$, and $g^{\rho}$ all H{\"o}lder continuous in time with exponent $1/2$ then 
\begin{align*}
& E \max_{i \leq n} \sum_{|\lambda| \leq m + 1} \sum_{x \in G_{h}} | \delta_{h,\lambda}( v^{h}_{i}(x) - u_{i}(x)) |^{2} h^{d} \\
& + E \tau \sum_{i=1}^{n} \sum_{|\lambda| \leq m + 2} \sum_{x \in G_{h}} | \delta_{h,\lambda}( v^{h}_{i}(x) - u_{i}(x)) |^{2} h^{d} \leq N (h^{2} + \tau)
\end{align*}
for sufficiently small $\tau$, $h \in (0,1)$, and for a constant $N$ that is independent of $h$ and $\tau$. The principal interest of this paper is to investigate higher order convergence with respect to the spatial discretization, that is, to obtain an estimate, similar to the above, with a higher power of $h$ by applying Richardson's method. 

While it is natural to seek solutions to \eqref{eqn: implicit space-time scheme} on the grid, carrying out our analysis on the whole space will have certain advantages when it comes to providing estimates for solutions to our schemes. Indeed, we observe that \eqref{eqn: implicit space-time scheme} is well defined not only on $G_{h}$ but for all $x \in \mathbf{R}^{d}$. Therefore, we introduce an alternate notion of solution. A solution to \eqref{eqn: implicit space-time scheme} on $\Omega \times T_{\tau} \times \mathbf{R}^{d}$ with an $L^{2}$-valued $\mathcal{F}_{0}$-measurable initial condition $v^{h}_{0}$ is a sequence $v^{h} = (v_{i}^{h})_{i=1}^{n}$ of $L^{2}$-valued random variables satisfying \eqref{eqn: implicit space-time scheme}. In a similar spirit, solutions to \eqref{eqn: implicit time scheme} with the appropriate initial condition are understood as sequences of $W^{1}_{2}$-valued random variables satisfying \eqref{eqn: implicit time scheme} in $W^{-1}_{2}$. The next result follows immediately from the considerations in the proof of Theorem \ref{thm: l2 valued solution to the space-time scheme}.

\begin{thm}
\label{thm: L2 valued solution to the space-time scheme}
Let $f$ and $g^{\rho}$ be $\mathcal{F}_{i}$-adapted $L^{2}$-valued processes and let $v^{h}_{0}$ be an $\mathcal{F}_{0}$-measurable $L^{2}$-valued initial condition. If Assumption \ref{asm: boundedness of space-time scheme coefficients} holds then \eqref{eqn: implicit space-time scheme} admits a unique $L^{2}$-valued solution for sufficiently small $\tau$. 
\end{thm}

By Sobolev's embedding theorem, for $l > d/2$ there exists a linear operator $I : W^{l}_{2} \to C_{b}$ such that $\phi(x) = I \phi (x)$ for almost every $x \in \mathbf{R}^{d}$ and $\sup_{x\in \mathbf{R}^{d}} |I\phi(x)| \leq N \| \phi \|_{l}$ for all $\phi \in W^{l}_{2}$ where $N$ is a constant. We recall the following useful embedding of $W^{l}_{2} \subseteq \ell^{2}(G_{h})$ from \cite{GyongyKrylov:2010}.

\begin{lem}
\label{lem: embedding}
Let $l > d/2$ and $|h| \in (0,1)$. For all $\phi \in W^{l}_{2}$ the embedding
\begin{equation}
\label{eqn: embedding}
\sum_{x\in G_{h}} | I \phi(x)|^{2} |h|^{d} \leq N \| \phi \|_{l}^{2}
\end{equation}
holds for a constant $N$ that depends only on $d$ and $l$. 
\end{lem}

\begin{proof}
For $z \in \mathbf{R}^d$ let $B_{r}(x) := \{ x\in \mathbf{R}^{d} ; |x -z| < r \}$. By the embedding of $W^{l}_{2}$ into $C_{b}$, for $\phi \in C_{b}$ we have 
\begin{align*}
 |\phi(z)|^{2} &\leq \sup_{x \in B_{1}(0)} \phi^{2} (z + hx)\\
 	&\leq N \sum_{|\alpha| \leq l } h^{2|\alpha|} \int_{B_{1}(0)} |(D^{\alpha}\phi)(z+hx)|^{2}\, dx \\
	&\leq N \sum_{|\alpha| \leq l } |h|^{2|\alpha|-d} \int_{B_{h}(z)} |(D^{\alpha}\phi)(x)|^{2}\, dx\\
	&\leq N |h|^{-d} \sum_{|\alpha|\leq l} \int_{B_{h}(z)} |(D^{\alpha}\phi)(x)|^{2}\, dx
\end{align*}
for a constant $N$ depending only on $d$ and $l$ and thus
\begin{equation*}
|\phi|_{\ell^{2}(G_{h})}^{2} = \sum_{z\in G_{h}} |\phi(z)|^{2} |h|^{d} \leq N \sum_{|\alpha|\leq l} \sum_{z\in G_{h}} \int_{B_{h}(z)} |(D^{\alpha}\phi)(x)|^{2}\, dx,
\end{equation*}
which yields the desired embedding.
\end{proof}

We will show that the restriction of a continuous modification of an $L^{2}$-valued solution to \eqref{eqn: implicit space-time scheme} to the grid $G_{h}$ is also a solution in the $\ell^{2}(G_{h})$ sense. Thus we will carry out our analysis in the whole space and obtain estimates independent of $h$ in appropriate Sobolev spaces for the $L^{2}$-valued solutions of \eqref{eqn: implicit space-time scheme} and \eqref{eqn: implicit time scheme}.
 
We provide the aforementioned Sobolev space estimates in Section \ref{sec: Auxiliary Results}. We then use these estimates in Section \ref{sec: Proof of Main Results} to prove the main results, which are the focus of the next section.

\section{Main Results}\label{sec: Main Results}
To accelerate the convergence of the spatial approximation by Richardson's method we must have an expansion for the solution $v^{h}$ to \eqref{eqn: implicit space-time scheme} with initial data $v^{h}_{0} = u_{0}$ in powers of the mesh size $h$. This relies on the possibility of proving the existence of sequences of random fields $v^{(0)}(x)$, $v^{(1)}(x)$, \dots, $v^{(k)}(x)$, for $x \in \mathbf{R}^{d}$ and integer $k \geq 0$, satisfying certain properties. Namely,  $v^{(0)}, \dots, v^{(k)}$ are independent of $h$; $v^{(0)}$ is the solution of \eqref{eqn: implicit time scheme} with initial value $u_{0}$; and an expansion 
\begin{equation}
\label{eqn: expansion}
v^{h}_{i} (x) = \sum_{j=0}^{k} \frac{h^{j}}{j!} v^{(j)}_{i}(x) + R^{\tau,h}_{i}(x)
\end{equation}
holds almost surely for $i \in \{1, \dots, n\}$ and $x \in G_{h}$, where $R^{\tau,h}$ is an $\ell_{2}(G_{h})$-valued adapted process such that 
\begin{equation}
\label{eqn: estimate for expansion error}
E \max_{i \leq n} \sup_{x \in G_{h}} |R^{\tau,h}_{i}(x)|^{2} \leq N h^{2(k+1)}\mathcal{K}_{m}
\end{equation}
for $$\mathcal{K}_{m} :=  E \| u_{0} \|_{m+1}^{2} + E \tau \sum_{i = 0}^{n} ( \| f_{i} \|_{m}^{2} + \| g_{i} \|_{m+1}^{2} ) < \infty$$ and a constant $N$ independent of $\tau$ and $h$. 

Our first result concerns the existence of such an expansion.

\begin{thm}
\label{thm: expansion and estimate for the expansion error}
If Assumptions \ref{asm: boundedness of coefficients}, \ref{asm: strong stochastic parabolicity}, \ref{asm: initial conditions and free terms}, \ref{asm: boundedness of space-time scheme coefficients}, \ref{asm: stochastic parabolicity for space-time scheme}, and \ref{asm: consistency} hold with $$\mathfrak{m} = m > k + 1 + \frac{d}{2}$$ for an integer $k \geq 0$ then expansion \eqref{eqn: expansion} and estimate \eqref{eqn: estimate for expansion error} hold for a constant $N$ depending only on $d$, $d_{1}$, $\Lambda$, $m$, $K_{0}$, \dots, $K_{m+1}$, $A_{0}$, \dots, $A_{m}$, $\kappa$, and $T$. 
\end{thm}

In the proof of Theorem \ref{thm: expansion and estimate for the expansion error}, as $v^{h}$ is defined not only on $G_{h}$ but for all $x \in \mathbf{R}^{d}$, we will see that one can replace $G_{h}$ in  \eqref{eqn: estimate for expansion error} with $\mathbf{R}^{d}$. We also note that in the situation of Remark \ref{rmk: on equivalence of assumptions}, if Assumptions \ref{asm: boundedness of coefficients} and \ref{asm: strong stochastic parabolicity} hold with $m > k + 1 + d/2$ then the conditions of Theorem \ref{thm: expansion and estimate for the expansion error} are satisfied. 

Taking differences of expansion \eqref{eqn: expansion} clearly yields $$\delta_{h,\lambda} v^{h}_{i} (x) = \sum_{j=0}^{k} \frac{h^{j}}{j!} \delta_{h,\lambda} v^{(j)}_{i} (x) + \delta_{h,\lambda} R^{\tau, h}_{i} (x)$$ for any $\lambda = (\lambda_{1},\dots,\lambda_{p}) \in \Lambda^{p}$, for integer $p \geq 0$, where $\Lambda^{0} := \{0\}$ and $\delta_{h,\lambda} := \delta_{h,\lambda_{1}} \times \dots \times \delta_{h,\lambda_{p}}$. The bound on $\delta_{h,\lambda} R^{\tau, h}$ is not obvious, nevertheless we have the following generalization of the above theorem.

\begin{thm}
\label{thm: expansion and estimate for differences of expansion error}
If the assumptions of Theorem \ref{thm: expansion and estimate for the expansion error} hold with $$\mathfrak{m} = m > k + p + 1 + \frac{d}{2}$$ for a nonnegative integer $p$ then for $\lambda \in \Lambda^{p}$ expansion \eqref{eqn: expansion} and 
$$E \max_{i \leq n} \sup_{x \in G_{h}} |\delta_{h,\lambda}R^{\tau, h}_{i} (x)|^{2} + E \max_{i \leq n} |h|^{d}\sum_{x \in G_{h}} |\delta_{h,\lambda} R^{\tau, h}_{i}(x) |^{2} \leq N h^{2(k+1)}\mathcal{K}_{m}$$ hold for a constant $N$ depending only on $d$, $d_{1}$, $\Lambda$, $m$, $K_{0}$, \dots, $K_{m+1}$, $A_{0}$, \dots, $A_{m}$, $\kappa$, and $T$.
\end{thm}

The proof of Theorems \ref{thm: expansion and estimate for the expansion error} and \ref{thm: expansion and estimate for differences of expansion error} appear in Section \ref{sec: Proof of Main Results} following the considerations in the next section. Currently we shall discuss how to implement Richardson's method to obtain higher order convergence in the spatial approximation, extending the result from \cite{GyongyKrylov:2010} to the space-time scheme.

Fix an integer $k \geq 0$ and let 
\begin{equation}
\label{eqn: v-bar} 
\bar{v}^{h} := \sum_{j=0}^{k} \beta_{j} v^{2^{-j} h}
\end{equation}
where $v^{2^{-j}h}$ solves, with $2^{-j}h$ in place of $h$, the space-time scheme \eqref{eqn: implicit space-time scheme} with initial condition $u_{0}$. Here $\beta$ is given by $ (\beta_{0}, \beta_{1}, \dots, \beta_{k}) := (1, 0, \dots, 0) V^{-1} $
where $V^{-1}$ is the inverse of the Vandermonde matrix with entries $V^{ij} := 2^{-(i-1)(j-1)}$ for $i,j \in \{1, \dots, k+1\}$. Recall that $v^{(0)}$ is the solution to \eqref{eqn: implicit time scheme} with initial condition $u_{0}$.

\begin{thm}
\label{thm: acceleration}
Under the assumptions of Theorem \ref{thm: expansion and estimate for the expansion error}, 
\begin{equation}
\label{eqn: acceleration}
E \max_{i \leq n} \sup_{x \in G_{h}} | \bar{v}^{h}_{i} (x) - v^{(0)}_{i}(x) |^2 \leq N |h|^{2(k+1)} \mathcal{K}_{m}
\end{equation}
for a constant $N$ depending only on $d$, $d_{1}$, $\Lambda$, $m$, $K_{0}$, \dots, $K_{m+1}$, $A_{0}$, \dots, $A_{m}$, $\kappa$, and $T$. 
\end{thm}

\begin{proof}
By Theorem \ref{thm: expansion and estimate for the expansion error} we have the expansion 
$$v^{2^{-j}h} = v^{(0)} + \sum_{i=1}^{k} \frac{h^{i}}{i! 2^{ij}} v^{(i)} + \hat{r}^{\tau, 2^{-j}h} h^{k+1}$$
for each $j \in \{0,1,\dots, k\}$ where $\hat{r}^{\tau,2^{-j}h} := h^{-(k+1)} R^{\tau, 2^{-j}h}$. Then 
\begin{align*}
\bar{v}^{h} 
	&= \left( \sum_{j=0}^{k} \beta_{j} \right) v^{(0)} + \sum_{j=0}^{k} \sum_{i=1}^{k} \beta_{j} \frac{h^{i}}{i!2^{ij}} v^{(i)} + \sum_{j=0}^{k} \beta_{j} \hat{r}^{\tau, 2^{-j}h} h^{k+1}\\
	&= v^{(0)} + \sum_{i=1}^{k} \frac{h^{i}}{i!} v^{(i)} \sum_{j=0}^{k} \frac{\beta_{j}}{2^{ij}} + \sum_{j=0}^{k} \beta_{j} \hat{r}^{\tau, 2^{-j}h}\\
	&= v^{(0)} + \sum_{j=0}^{k} \beta_{j} \hat{r}^{\tau, 2^{-j}h}h^{k+1}
\end{align*}
since $\sum_{j=0}^{k} \beta_{j} = 1$ and $\sum_{j=0}^{k} \beta_{j} 2^{-ij} = 0$ for each $i \in \{ 1,2, \dots, k\}$ by the definition of $(\beta_{0}, \dots, \beta_{k})$. Now using the bound on $R^{\tau,h}$ from Theorem \ref{thm: expansion and estimate for the expansion error} together with this last calculation yields the desired result. 
\end{proof}

One can also construct rapidly converging approximations of derivatives of $v^{(0)}$. That is, if the conditions of Theorem \ref{thm: expansion and estimate for the expansion error} hold instead with $$\mathfrak{m} = m > k + p + 1 + \frac{d}{2}$$ for nonnegative integers $k$ and $p$ then  Theorem \ref{thm: acceleration} holds with $\delta_{h,\lambda}\bar{v}^{h}$ and $\delta_{h,\lambda}v^{(0)}$ in place of $\bar{v}^{h}$ and $v^{(0)}$, respectively, for $\lambda \in \Lambda^{p}$. Therefore, using suitable linear combinations of finite differences of $\bar{v}^{h}$ one can construct rapidly converging approximations for the derivatives of $v^{(0)}$. 

In the next section, we present material that will be used to prove the main results in this section. In particular, we provide estimates for the $L^{2}$-valued solutions of \eqref{eqn: implicit space-time scheme} and \eqref{eqn: implicit time scheme} in appropriate Sobolev spaces.

\section{Auxiliary Results}\label{sec: Auxiliary Results}

We include the following bound, which is given for the continuous time case in \cite{GyongyKrylov:2010}, for the convenience of the reader.

\begin{lem}
\label{lem: bound on Q}
If Assumptions \ref{asm: boundedness of space-time scheme coefficients} and \ref{asm: stochastic parabolicity for space-time scheme}  hold then for all $\phi \in L^{2}$
\begin{equation*}
\begin{split}
Q_{i}(\phi) := \int_{\mathbf{R}^{d}} 2 \phi(x) L^{h}_{i}\phi(x) + \sum_{\rho=1}^{d_{1}} |M^{h,\rho}_{i} \phi(x)|^{2} \, dx \\ \leq N \|\phi\|_{0}^{2} -\frac{\kappa}{2} \sum_{\lambda \in \Lambda_{0}} \| \delta_{h,\lambda} \phi \|_{0}^{2} 
\end{split}
\end{equation*}
for all $i \in \{1, \dots, n\}$ and for a constant $N$ depending only on $\kappa$, $A_{0}$, $A_{1}$, and the cardinality of $\Lambda$.
\end{lem}

\begin{proof}
First observe that for $\mu \in \Lambda_{0}$ the conjugate operator in $L_{2}$ to $\delta_{-h,\mu}$ is $-\delta_{h,\mu}$.
Notice also that $$\delta_{h,\mu} (\phi \psi) = \phi \delta_{h,\mu}\psi + (T_{h,\mu}\psi)\delta_{h,\mu}\phi$$ where $T_{h,\mu} \psi(x) = \psi (x+h\mu)$. Thus by simple calculations $Q = Q^{(1)} + Q^{(2)} + Q^{(3)} + Q^{(4)}$ where
$$Q^{(1)}_{i}(\phi) := - \int_{\mathbf{R}^{d}} \sum_{\lambda, \mu \in \Lambda_{0}} ((2\mathfrak{a}^{\lambda \mu}_{i} - \mathfrak{b}^{\lambda \rho}_{i} \mathfrak{b}^{\mu \rho}_{i})(\delta_{h, \lambda}\phi) \delta_{h,\mu}\phi)(x)\, dx,$$
$$Q^{(2)}_{i}(\phi) := - 2\int_{\mathbf{R}^{d}} \sum_{\lambda, \mu \in \Lambda_{0}} ((T_{h,\mu}\phi)(\delta_{h,\lambda}\phi)\delta_{h,\mu}\mathfrak{a}^{\lambda\mu}_{i})(x)\, dx,$$
$$Q^{(3)}_{i}(\phi) := 2 \int_{\mathbf{R}^{d}} (\mathfrak{a}^{00}_{i}\phi^{2}(x) + \phi(x) \sum_{\lambda \in \Lambda_{0}}(\mathfrak{a}^{\lambda 0}_{i} \delta_{h,\lambda} \phi + \mathfrak{a}^{0\lambda}_{i} \delta_{-h,\lambda} \phi)(x))\, dx,$$ and
$$Q^{(4)}_{i}(\phi) := \int_{\mathbf{R}^{d}} (\mathfrak{b}^{00}_{i} \phi^{2}(x) + 2 \sum_{\lambda \in \Lambda_{0}} \mathfrak{b}^{\lambda \rho}_{i} \mathfrak{b}^{0 \rho}_{i} \phi \delta_{h,\lambda} \phi(x))\, dx.$$

By Assumption \ref{asm: stochastic parabolicity for space-time scheme}, 
$$Q^{(1)}_{i}(\phi) \leq -\kappa \sum_{\lambda \in \Lambda_{0}} \| \delta_{h,\lambda}\phi\|_{0}^{2}$$ and by Assumption \ref{asm: boundedness of space-time scheme coefficients}, Young's inequality, and the shift invariance of Lebesgue measure, 
$$Q^{(j)}_{i}(\phi) \leq \frac{\kappa}{6} \sum_{\lambda \in \Lambda_{0}} \| \delta_{h,\lambda}\phi\|_{0}^{2} + N \| \phi \|_{0}^{2}$$ for each $j \in \{2,3,4\}$ with a constant $N$ depending only on the cardinality of $\Lambda$, $\kappa$, $A_{0}$ and, for $j=2$, also on $A_{1}$. 
\end{proof}

We also recall the following discrete Gronwall lemma. Note that we use the convention that summation over an empty set is zero.

\begin{lem}
\label{lem: discrete Gronwall lemma}
For constants $K \in (0, 1)$ and $C$, if $(a_{i})_{i=0}^{n}$ is a nonnegative sequence such that $a_{j}\leq C + K \sum_{i=1}^{j} a_{i}$ holds for each $j \in \{0, \dots, n\}$ then $a_{j} \leq C (1-K)^{-j}$ for $j \in \{0, \dots, n\}$. 
\end{lem}

\begin{proof}
Let $b_{j} = C + K \sum_{i=1}^{j}b_{i}$ and note that $(1-K)b_{j} = b_{j-1}$. Then $a_{j} \leq b_{j}$ for $j \leq n$ by induction, since $a_{0} \leq C = b_{0}$ and $$a_{j}(1-K) \leq C + \sum_{i=1}^{j-1}a_{i} \leq C + \sum_{i=1}^{j-1}b_{i} = b_{j} (1-K),$$ assuming $a_{j-1} \leq b_{j-1}$. Therefore $a_{j} \leq b_{j} = C (1-K)^{-j}$ for each $j \leq n$ and for $K \in (0,1)$.
\end{proof}

The following provides a Sobolev space estimate for solutions to the space-time scheme that is independent of $h$. For an integer $m \geq 0$, denote by $\mathbf{W}^{m}_{2}(\tau)$ the space of $W^{m}_{2}$-valued predictable processes $\phi$ on $\Omega \times T_{\tau}$ such that $$ [\![ \phi ]\!]_{m} := E \tau \sum_{i=1}^{n} \| \phi_{i} \|_{m}^{2} < \infty$$ 
and note that we write $$[\![ g ]\!]_{m} = E \tau \sum_{i=1}^{n} \sum_{\rho=1}^{d_{1}} \| g^{\rho}_{i} \|_{m}^{2}$$ for functions $g = (g^{\rho})_{\rho=1}^{d_{1}}$.

\begin{thm}
\label{thm: Sobolev valued solution to space-time scheme and estimate}
For $\mu \in \Lambda$ and $\rho \in \{1, \dots, d_{1}\}$, let $f^{\mu}, g^{\rho} \in \mathbf{W}_{2}^{\mathfrak{m}}(\tau)$. If Assumption \ref{asm: boundedness of space-time scheme coefficients} holds then for each nonzero $h$ there exists a unique solution $\nu \in \mathbf{W}_{2}^{\mathfrak{m}}(\tau)$ of 
\begin{equation}
\label{eqn: modified space-time scheme}
\nu_{i} = \nu_{i-1} + \sum_{\mu \in \Lambda} (L^{h}_{i}\nu_{i} + f^{\mu}_{i}) \tau + \sum_{\rho=1}^{d_{1}}( M^{h,\rho}_{i-1} \nu_{i-1} + g^{\rho}_{i-1}) \xi^{\rho}_{i}
\end{equation}
for any $W_{2}^{\mathfrak{m}+1}$-valued $\mathcal{F}_{0}$-measurable initial condition $\nu_{0}$. Further, if Assumption \ref{asm: stochastic parabolicity for space-time scheme} is also satisfied then
\begin{equation}
\label{eqn: modified space-time scheme estimate}
\begin{split}
E \max_{i \leq n} \| \nu_{i} \|_{\mathfrak{m}}^{2} + E \tau \sum_{i = 1}^{n} \sum_{\lambda \in \Lambda} \| \delta_{h,\lambda} \nu_{i} \|_{\mathfrak{m}}^{2} \leq N E \tau \|\nu_{0}\|_{\mathfrak{m}+1}^{2} \\ + N E \tau \sum_{i=0}^{n} ( \| f_{i} \|_{\mathfrak{m}}^{2} + \| g_{i} \|_{\mathfrak{m}}^{2} )
\end{split}
\end{equation}
holds for a constant $N$ that depends only on $d$, $d_{1}$, $\mathfrak{m}$, $\Lambda$, $A_{0}$, \dots, $A_{\bar{\mathfrak{m}}}$, $\kappa$, and $T$.
\end{thm}

\begin{proof}
By Theorem \ref{thm: L2 valued solution to the space-time scheme}, the existence of a unique sequence of $L^{2}$-valued random variables solving \eqref{eqn: implicit space-time scheme} is known. For $f^{\mu}, g^{\rho} \in \mathbf{W}^{\mathfrak{m}}_{2}(\tau)$ and an $W^{\mathfrak{m}+1}_{2}$-valued initial condition $\nu_{0}$, there exists a unique sequence of $W^{\mathfrak{m}}_{2}$-valued random variables satisfying \eqref{eqn: modified space-time scheme}. The estimate \eqref{eqn: modified space-time scheme estimate} can be achieved easily with a constant $N$ depending on $h$, so in particular the solution is in $\mathbf{W}^{\mathfrak{m}}_{2}(\tau)$. 

Next we prove the estimate \eqref{eqn: modified space-time scheme estimate} for a constant independent of $h$. For convenience we denote $\mathfrak{K}_{\mathfrak{m}}^{n} := \tau \sum_{i=0}^{n} ( \| f_{i} \|_{\mathfrak{m}}^{2} + \| g_{i} \|_{\mathfrak{m}}^{2} ).$ Considering equalities of the from $a^{2} + b^{2} = 2a(a-b) - |a-b|^{2}$ we note that \eqref{eqn: modified space-time scheme} implies
\begin{align*}
\|\nu_{i}\|_{0}^{2} - \|\nu_{i-1} \|_{0}^{2} &= 2 (\nu_{i}, \nu_{i} - \nu_{i-1}) - \|\nu_{i} - \nu_{i-1}\|_{0}^{2}\\
	& = 2(\nu_{i}, L_{i}^{h}\nu_{i} + f^{\mu}_{i})\tau + 2(\nu_{i-1}, M^{h,\rho}_{i-1}\nu_{i-1} + g^{\rho}_{i-1})\xi^{\rho}_{i}\\
	& \qquad + 2(\nu_{i} - \nu_{i-1},M^{h,\rho}_{i-1}\nu_{i-1} + g^{\rho}_{i-1})\xi^{\rho}_{i} - \|\nu_{i} - \nu_{i-1} \|_{0}^{2}\\
	&= 2(\nu_{i}, L^{h}_{i}\nu_{i} + f^{\mu}_{i})\tau + 2(\nu_{i-1}, M^{h,\rho}_{i-1}\nu_{i-1} + g^{\rho}_{i-1})\xi^{\rho}_{i}\\
	& \qquad + \|(M^{h,\rho}_{i-1}\nu_{i-1} + g^{\rho}_{i-1})\xi^{\rho}_{i}\|_{0}^{2} - \|L^{h}_{i} \nu_{i} + f^{\mu}_{i}\|_{0}^{2} \tau^{2}
\end{align*} 
where here and in what follows we suppress the sums over $\mu \in \Lambda_{0}$ and $\rho \in \{1, \dots, d_{1}\}$. Summing up over $i$, we have 
\begin{equation}
\label{eqn: intermediate equation with H, I, J}
\| \nu_{j} \|_{0}^{2} \leq \| \nu_{0}\|_{0}^{2} + \mathcal{H}_{j} + \mathcal{I}_{j} + \mathcal{J}_{j}
\end{equation}
where 
\begin{equation*}
\mathcal{H}_{j} := \sum_{i=1}^{j} 2(\nu_{i}, L^{h}_{i} \nu_{i} + f^{\mu}_{i} )\tau,
\end{equation*}
\begin{equation*}
\mathcal{I}_{j} := \sum_{i=1}^{j} 2(\nu_{i-1}, M^{h,\rho}_{i-1} \nu_{i-1} + g^{\rho}_{i-1} )\xi^{\rho}_{i},
\end{equation*}
and 
\begin{equation*}
\mathcal{J}_{j} := \sum_{i=1}^{j} \| (M^{h,\rho}_{i-1} \nu_{i-1} + g^{\rho}_{i-1}) \xi^{\rho}_{i} \|_{0}^{2}.
\end{equation*}

By an application of It{\^o}'s formula, it is easy to see that for $\pi,\rho \in \{1, \dots, d_{1}\}$
\begin{equation*}
\xi^{\pi}_{i+1}\xi^{\rho}_{i+1} = (\Delta w^{\pi} (t_{i})) (\Delta w^{\rho}(t_{i})) = Y^{\pi \rho}_{i+1} - Y^{\pi \rho}_{i} + \tau \delta_{\pi \rho}
\end{equation*}
for all $i \in \{1, \dots, n\}$ where 
$$Y^{\pi \rho}(t) := \int_{0}^{t} (w^{\pi}(s) - w^{\pi}_{\gamma(s)})\, dw^{\rho}(s) + \int_{0}^{t} (w^{\rho}(s) - w^{\rho}_{\gamma (s)}) \, dw^{\pi}(s),$$ 
$\gamma (s)$ is the piecewise defined function taking value $\gamma(s) = i$ for $s \in [i\tau,(i+1)\tau)$, and $\delta_{\pi\rho} = 1$ when $\pi = \rho$ and $0$ otherwise. Thus can write $\mathcal{J}_{j} = \mathcal{J}^{(1)}_{j} + \mathcal{J}^{(2)}_{j}$ where
\begin{equation*}
\mathcal{J}^{(1)}_{j} := \sum_{i=1}^{j} \sum_{\rho=1}^{d_{1}} \| M^{h,\rho}_{i-1} \nu_{i-1} + g^{\rho}_{i-1} \|_{0}^{2} \tau
\end{equation*}
and
\begin{equation*}
\mathcal{J}^{(2)}_{j} := \int_{0}^{t_{j}} \sum_{\pi,\rho = 1}^{d_{1}} (M^{h,\pi}_{\gamma (s)} \nu_{\gamma (s)} + g^{\pi}_{\gamma (s)}, M^{h,\rho}_{\gamma (s)} \nu_{\gamma (s)} + g^{\rho}_{\gamma (s)}) \, d Y^{\pi \rho}(s).
\end{equation*}
Then we note that since Lemma \ref{lem: bound on Q} holds for all $t \in [0,T]$, in particular
\begin{align*}
\mathcal{H}_{j} + \mathcal{J}^{(1)}_{j} 
	&\leq \tau \sum_{\lambda \in \Lambda} \| \delta_{h,\lambda} \nu_{0}\|_{0}^{2} + \tau \sum_{i=1}^{j} \left( Q_{i} (\nu_{i}) + (\nu_{i}, f^{\mu}_{i}) + \frac{\tau}{\varepsilon} \|g_{i-1}\|_{0}^{2}\right) \\
	&\leq C \tau \| \nu_{0}\|_{1}^{2} + N \tau \sum_{i=1}^{j} \|\nu_{i}\|_{0}^{2} - \frac{\kappa}{2} \tau \sum_{i=1}^{j} \sum_{\lambda \in \Lambda} \|\delta_{h,\lambda} \nu_{i}\|_{0}^{2} + N \mathfrak{K}_{0}^{j}
\end{align*}
where here $\| \nu_{0}\|_{0}^{2} \leq C \| \nu_{0} \|_{1}^{2}$ and $N$ is a positive constant depending only on $\kappa$, $A_{0}$, $A_{1}$, and the cardinality of $\Lambda$. Thus we can replace equation \eqref{eqn: intermediate equation with H, I, J} by 
\begin{equation}
\label{eqn: intermediate equation after applying Q bound}
\begin{split}
\| \nu_{j} \|_{0}^{2} + \tau \sum_{i=1}^{j} \sum_{\lambda \in \Lambda} \|\delta_{h,\lambda} \nu_{i}\|_{0}^{2} \leq N \tau \|\nu_{0}\|_{1}^{2} + N \tau \sum_{i=1}^{j} \|\nu_{i}\|_{0}^{2} \\ + N \mathfrak{K}_{0}^{j} + \mathcal{I}_{j} + \mathcal{J}^{(2)}_{j} \end{split}
\end{equation}
for a constant $N$ that depends only on $\kappa$, $A_{0}$, $A_{1}$, $C$, and the cardinality of $\Lambda$.

Next we observe that 
$$ E \mathcal{I}_{j} = \sum_{i=1}^{j} 2 \int_{\mathbf{R}^{d}} E \left( E \left( \nu_{i-1} (M^{h,\rho}_{i-1} \nu_{i-1} + g^{\rho}_{i-1}) \xi^{\rho}_{i} \,|\, \mathcal{F}_{i-1} \right) \right) \, dx = 0$$ 
since $\xi^{\rho}_{i+1}$ is independent of $\mathcal{F}_{i}$ and $\nu_{i}$, $M^{h,\rho}_{i} \nu_{i}$, and $g^{\rho}_{i}$ are all $\mathcal{F}_{i}$-measurable for $i \in \{0, \dots, n\}$. Similarly, we see that $E \mathcal{J}^{(2)}_{j} = 0$ since the expectation of the stochastic integral is zero. Therefore, taking the expectation of \eqref{eqn: intermediate equation after applying Q bound} we have that
\begin{equation}
\label{eqn: intermediate equation after taking the expectation and maximum}
\begin{split}
E \| \nu_{j} \|_{0}^{2}  + E \tau \sum_{i=1}^{n} \sum_{\lambda \in \Lambda} \|\delta_{h,\lambda} \nu_{i}\|_{0}^{2} \leq N E \tau \|\nu_{0}\|_{1}^{2} + N E \mathfrak{K}_{0}^{n} \\ + N E \tau \sum_{i=1}^{j} \| \nu_{i} \|_{0}^{2}
\end{split}
\end{equation}
for each $j \in \{1, \dots, n\}$. Excluding for the time being the difference term on the left hand side of \eqref{eqn: intermediate equation after taking the expectation and maximum} and applying Lemma \ref{lem: discrete Gronwall lemma} we obtain
$$ E \| \nu_{j}\|_{0}^{2} \leq N (E \tau \| \nu_{0}\|_{1}^{2} + E \mathfrak{K}_{0}^{n}) (1 - N \tau)^{-j}$$
and, since $(1 - N\tau)^{-j} = (1- N \frac{T}{n})^{-j} \leq (1 - N \frac{T}{n})^{-n} \leq C^{\prime} e^{NT}$, we have 
\begin{equation}
\label{eqn: intermediate bound on max E v after applying discrete Gronwall lemma}
\max_{i \leq n} E \| \nu_{i}\|_{0}^{2} \leq N E \tau \| \nu_{0} \|_{1}^{2} + N E \mathfrak{K}_{0}^{n}
\end{equation}
for a constant $N$ here that depends only on the parameters $\kappa$, $A_{0}$, $A_{1}$, $T$, and the cardinality of $\Lambda$. Using equation \eqref{eqn: intermediate bound on max E v after applying discrete Gronwall lemma} we can eliminate the last term on the right hand side of \eqref{eqn: intermediate equation after taking the expectation and maximum}. In particular, we have 
\begin{equation*}
E \tau \sum_{i=1}^{n} \sum_{\lambda \in \Lambda} \| \delta_{h,\lambda} \nu_{i} \|_{0}^{2} \leq N E \tau \| \nu_{0} \|_{1}^{2} + N E \mathfrak{K}_{0}^{n}.
\end{equation*}

Using the Davis inequality we can bound $\max |\mathcal{J}^{(2)}|$ and $\max |\mathcal{I}|$. Namely,
\begin{align*}
E &\max_{i \leq n} |\mathcal{J}^{(2)}_{j}| \\
	&\leq 3 \sum_{\pi,\rho=1}^{d_{1}} E \left\{ \int_{0}^{T} \|M^{h,\rho}_{\gamma (s)} \nu_{\gamma (s)} + g^{\rho}_{\gamma (s)}\|_{0}^{2} \|M^{h,\pi}_{\gamma (s)}\nu_{\gamma (s)} + g^{\pi}_{\gamma (s)}\|_{0}^{2} \, d \langle Y^{\pi\rho}\rangle(s)\right\}^{1/2}\\
	&\leq C \sum_{\pi,\rho=1}^{d_{1}} E \left\{ \int_{0}^{T} \|M^{h,\rho}_{\gamma (s)} \nu_{\gamma (s)} + g^{\rho}_{\gamma (s)}\|_{0}^{4} |w^{\pi}(s) - w^{\pi}_{\gamma (s)}|^{2} \, ds \right\}^{1/2}\\
	&\leq C \sum_{\pi,\rho=1}^{d_{1}} E \left(\max_{i \leq n} \sqrt{\tau} \|M^{h,\rho}_{i} \nu_{i} + g^{\rho}_{i}\|_{0}\right. \\
		&\qquad \left. \times \left\{ \frac{1}{\tau} \int_{0}^{T} \|M^{h,\rho}_{\gamma (s)} \nu_{\gamma (s)} + g^{\rho}_{\gamma (s)}\|_{0}^{2} |w^{\pi}(s) - w^{\pi}_{\gamma (s)} |^{2}\, ds \right\}^{1/2}\right)
\end{align*}
where $C$ is a constant independent of $\tau$ and $h$ that is allowed to change from one instance to the next. Therefore,
\begin{equation}
\label{eqn: intermediate estimate for E max J-2}
\begin{split}
E \max_{i \leq n} | \mathcal{J}^{(2)}_{i} | \leq d_{1} C \sum_{\rho=1}^{d_{1}} \tau E \max_{i \leq n} \| M^{h,\rho}_{i}\nu_{i} + g^{\rho}_{i}\|_{0}^{2} \\ + \frac{C}{\tau} \sum_{\pi,\rho=1}^{d_{1}} E \int_{0}^{T} \| M^{h,\rho}_{\gamma (s)} \nu_{\gamma (s)} + g^{\rho}_{\gamma (s)}\|_{0}^{2} |w^{\pi}(s) - w^{\pi}_{\gamma (s)}|^{2} \, ds
\end{split}
\end{equation}
by Young's inequality. The first term on the right hand side of \eqref{eqn: intermediate estimate for E max J-2} is bounded from above by the sum over all $i \in \{1, \dots, n\}$, hence
\begin{align*}
\sum_{\rho=1}^{d_{1}} \tau E \max_{i \leq n} \| M^{h,\rho}_{i} \nu_{i} + g^{\rho}_{i} \|_{0}^{2} &\leq E \tau \sum_{i=0}^{n} \| M^{h,\rho}_{i} \nu_{i} + g^{\rho}_{i} \|_{0}^{2}\\
	&\leq C E \tau \sum_{i=0}^{n} ( \sum_{\lambda \in \Lambda} \| \delta_{h,\lambda} \nu_{i} \|_{0}^{2} + \| g_{i} \|_{0}^{2} ),
\end{align*}
and the second term on the right hand side of \eqref{eqn: intermediate estimate for E max J-2} yields
\begin{align*}
\frac{1}{\tau}& \sum_{\pi, \rho=1}^{d_{1}} E \int_{0}^{T} \| M^{h,\rho}_{\gamma (s)} \nu_{\gamma (s)} + g^{\rho}_{\gamma (s)} \|_{0}^{2} | w^{\pi}(s) - w^{\pi}_{\gamma (s)}|^{2} \, ds \\
	&\leq \frac{1}{\tau} \sum_{\pi, \rho=1}^{d_{1}} E \left( E \left( \int_{0}^{T} \| M^{h,\rho}_{\gamma (s)} \nu_{\gamma (s)} + g^{\rho}_{\gamma (s)} \|_{0}^{2} | w^{\pi}(s) - w^{\pi}_{\gamma (s)}|^{2} \, ds \,|\, \mathcal{F}_{\gamma (s)} \right) \right)\\
	&\leq E \tau \sum_{i=0}^{n} ( \sum_{\lambda \in \Lambda} \| \delta_{h,\lambda}\nu_{i} \|_{0}^{2} + \| g_{i} \|_{0}^{2} )
\end{align*}
by the tower property for conditional expectations. Combining these estimates we see that $E \max | \mathcal{J}^{(2)} |$ is estimated by terms already appearing on the right hand side of \eqref{eqn: intermediate equation after applying Q bound} for a constant $N$ that depends also on $d_{1}$. 

Similarly, we note that
\begin{align*}
\mathcal{I}_{j} &= \sum_{i=1}^{j} 2 (\nu_{i-1}, M^{h,\rho}_{i-1} \nu_{i-1} + g^{\rho}_{i-1} ) \Delta w^{\rho}_{i-1} \\
	&= 2 \int_{0}^{t_{j}} (\nu_{\gamma (s)}, M^{h,\rho}_{\gamma (s)} \nu_{\gamma (s)} + g^{\rho}_{\gamma (s)}) \, dw^{\rho}(s).
\end{align*}
Applying the Davis inequality
\begin{align*}
E \max_{i \leq n} | \mathcal{I}_{i} | 
	&\leq 6 \sum_{\rho=1}^{d_{1}} E \left\{ \int_{0}^{T} \| \nu_{\gamma (s)} \|_{0}^{2} \| M^{h,\rho}_{\gamma (s)} \nu_{\gamma (s)} + g^{\rho}_{\gamma (s)} \|_{0}^{2} \, ds \right\}^{1/2}\\
	&\leq 6 \sum_{\rho=1}^{d_{1}} E \left( \max_{i \leq n}\|\nu_{i}\|_{0} \left\{ \int_{0}^{T} \| M^{h,\rho}_{\gamma (s)}\nu_{\gamma (s)} + g^{\rho}_{\gamma (s)} \|_{0}^{2} \, ds \right\}^{1/2}\right)
\end{align*}
and then Young's inequality
\begin{equation}
\label{eqn: estimate for E max I}
E \max_{i \leq n} | \mathcal{I}_{i} | \leq \frac{1}{2} E \max_{i \leq n}\| \nu_{i} \|_{0}^{2} + C E \tau \sum_{i= 0}^{n} (\sum_{\lambda \in \Lambda} \| \delta_{h,\lambda} \nu_{i} \|_{0}^{2} + \|g_{i} \|_{0}^{2})
\end{equation}
we see that $E \max | \mathcal{I} |$ is also estimated by terms already appearing on the right and side of \eqref{eqn: intermediate equation after applying Q bound}. 

Returning to \eqref{eqn: intermediate equation after applying Q bound} and taking the maximum followed by the expectation we have 
\begin{equation*}
E \max_{i \leq n} \| \nu_{i} \|_{0}^{2} + E \tau \sum_{i=1}^{n} \sum_{\lambda \in \Lambda} \| \delta_{h,\lambda} \nu_{i} \|_{0}^{2} \leq N E \tau \| \nu_{0} \|_{1}^{2} + N E \mathfrak{K}_{0}^{n},
\end{equation*}
using \eqref{eqn: intermediate bound on max E v after applying discrete Gronwall lemma} and the estimates on $E \max | \mathcal{J}^{(2)}|$ and $E \max | \mathcal{I}|$. Thus \eqref{eqn: modified space-time scheme estimate} holds when $\mathfrak{m}=0$. 

If $\mathfrak{m} \geq 1$ we differentiate \eqref{eqn: modified space-time scheme} with respect to $x^{l}$ and introduce the notation $\tilde{\phi}$ for the derivative of a function $\phi$ in the direction $x^{l}$ for $l \in \{1, \dots, d\}$. Then \eqref{eqn: modified space-time scheme} becomes
\begin{equation}
\label{eqn: modified space time discretized Cauchy problem with one derivative in x-i}
\begin{split}
\tilde{\nu}_{i} = \tilde{\nu}_{i-1} + \sum_{\lambda, \mu \in \Lambda} (\mathfrak{a}^{\lambda \mu}_{i}\delta_{h,\lambda}\delta_{-h,\mu}\tilde{\nu}_{i} + \hat{f}^{\mu}_{i}) \tau \\ + \sum_{\rho=1}^{d_{1}} \sum_{\lambda \in \Lambda}( \mathfrak{b}^{\lambda \rho}_{i-1} \delta_{h,\lambda}\tilde{\nu}_{i-1} + \hat{g}^{\rho}_{i-1} ) \xi^{\rho}_{i}
\end{split}
\end{equation}
where $\hat{f}^{\mu} := \tilde{f}^{\mu}$ for nonzero $\mu$, $\hat{f}^{0} := \tilde{f}^{0} + \tilde{\mathfrak{a}}^{\lambda\mu} \delta_{h,\lambda}\delta_{-h,\mu} \nu$ and $\hat{g}^{\rho} := \tilde{g}^{\rho} + \tilde{\mathfrak{b}}^{\lambda\rho} \delta_{h,\lambda} \nu$. Recalling that $\partial_{\mu} = \mu^{l} D_{l}$ for $\mu \in \Lambda_{0}$, we proceed as before but now using the inequality
\begin{align*}
E \tau \sum_{i=1}^{n} (\tilde{\nu}_{i}, \delta_{h,\lambda} \delta_{-h,\mu} \nu_{i}) &\leq \tau \sum_{i=1}^{n} E \|\tilde{\nu}_{i} \|_{0} \| \delta_{h,\lambda} \partial_{\mu} \nu_{i} \|_{0} \\
	&\leq \varepsilon E \tau \sum_{i=1}^{n} \|D \delta_{h,\lambda} \nu_{i} \|_{0}^{2} + \frac{N}{\varepsilon} E \tau \sum_{i=1}^{n} \| \tilde{\nu}_{i}\|_{0}^{2}
\end{align*}
which holds for arbitrary $\varepsilon > 0$ and $N$ depending only on $|\mu|$. This leads to the following 
\begin{equation}
\label{eqn: almost W1 estimate for modified space time scheme}
\begin{split}
E \| \tilde{\nu}_{j} \|_{0}^{2} + E \tau \sum_{i= 1}^{j} \sum_{\lambda \in \Lambda} \|\delta_{h,\lambda} \nu_{i} \|_{0}^{2} \leq N E \tau \| \nu_{0}\|_{1}^{2} + NE \mathfrak{K}_{1}^{n}\\
	+ \frac{1}{2d} E \tau \sum_{i =1}^{j} \sum_{\lambda \in \Lambda} \| D \delta_{h,\lambda} \nu_{i} \|_{0}^{2} + N E \tau \sum_{i=1}^{j}  \|\tilde{\nu}_{i} \|_{0}^{2}
\end{split}
\end{equation}
for each $x^{l}$, $l \in \{1, \dots, d\}$. Summing up over each direction $x^{l}$, the term with factor $1/2d$ can be seen to be estimated by other terms already appearing on the right hand side of \eqref{eqn: almost W1 estimate for modified space time scheme}. Then by the same procedure as before we obtain
\begin{equation}
\label{eqn: W1 estimate for modified space time scheme}
E \max_{i \leq n} \| D\nu_{i} \|_{0}^{2} + E \tau \sum_{i = 1}^{n} \sum_{\lambda \in \Lambda} \| D\delta_{h,\lambda} \nu_{i} \|_{0}^{2} \leq N E \tau \|\nu_{0}\|_{1}^{2} + N E \mathfrak{K}_{1}^{n}
\end{equation}
which proves the theorem when $\mathfrak{m}=1$.

Assuming that $\mathfrak{m} \geq 2$ and that \eqref{eqn: modified space-time scheme estimate} holds for each integer $p < \mathfrak{m}$ in place of $\mathfrak{m}$, then we can differentiate \eqref{eqn: modified space-time scheme} $(p+1)$ times and, repurposing the notation $\tilde{\phi}$ for the $(p+1)$th order derivatives of $\phi$ with respect to $x$, we obtain \eqref{eqn: modified space time discretized Cauchy problem with one derivative in x-i} with different $\hat{f}^{0}$ and $\hat{g}^{\rho}$. Namely, the $\hat{f}^{0}$ will be the sum of $\tilde{f}^{0}$ and linear combinations of certain $i$th order derivatives of $\mathfrak{a}^{\lambda\mu}$ together with certain $(p+1-i)$th order derivatives of $\delta_{h,\lambda}\delta_{-h,\mu} \nu$, for integer $i \leq (p+1)$. As before, the $L^{2}$-norms of the $(p+1-i)$th derivatives of $\delta_{h,\lambda}\delta_{-h,\mu} \nu$ are dominated by the $L^{2}$-norms of the $(p+2-i)$th derivatives of $\delta_{h,\lambda}\nu$ which are in turn less than the $W^{p+1}_{2}$-norm of $\delta_{h,\lambda} \nu$. After similar changes are made in $\hat{g}^{\rho}$ we obtain the counterpart of \eqref{eqn: W1 estimate for modified space time scheme} which then yields \eqref{eqn: modified space-time scheme estimate} with $(p+1)$ in place of $\mathfrak{m}$. 
\end{proof}

We also have the following Sobolev space estimate for solutions to the implicit time scheme. Let $m$ be a nonnegative integer.

\begin{thm}
\label{thm: solvability implicit time scheme with estimate}
Let $f \in \mathbf{W}^{m}_{2}(\tau)$ and $g^{\rho} \in \mathbf{W}^{m+1}_{2}(\tau)$. If Assumptions \ref{asm: boundedness of coefficients} and \ref{asm: strong stochastic parabolicity} hold then \eqref{eqn: implicit time scheme} has a unique solution $v \in \mathbf{W}^{m+2}_{2}(\tau)$ for a given $W^{m+1}_{2}$-valued $\mathcal{F}_{0}$-measurable initial condition $v_{0}$. Moreover
\begin{equation*}
\begin{split}
E \max_{i \leq n} \| v_{i}\|_{m+1}^{2} + E \tau \sum_{i=1}^{n} \| v_{i} \|_{m+2}^{2} \leq N E \| v_{0}\|_{m+1}^{2} \\ + N E \tau \sum_{i=0}^{n} (\| f_{i}\|_{m}^{2} + \| g_{i}\|_{m+1}^{2})
\end{split}
\end{equation*} 
holds for a constant $N$ depending only on $d$, $d_{1}$, $m$, $K_{0}$, \dots, $K_{m+1}$, $\kappa$, and $T$. 
\end{thm}

\begin{proof}
Proving the solvability of \eqref{eqn: implicit time scheme} reduces to solving the elliptic problem $$(I - \tau \mathcal{L}_{i})v_{i} = v_{i-1} + \tau f_{i} + \sum_{\rho=1}^{d_{1}} \xi^{\rho}_{i} ( \mathcal{M}^{\rho}_{i-1} v_{i-1} + g^{\rho}_{i-1})$$ for each $i \in \{1, \dots, n\}$ where $I$ is the identity. That is, we claim that $\mathcal{A}:= (I - \tau \mathcal{L})$ is a $W^{m}_{2}$-valued operator on $W^{m+2}_{2}$ such that $\mathcal{A}_{i}$ is
  \begin{enumerate}
  \item \emph{bounded}, i.e.\ $\| \mathcal{A}_{i} \phi \|_{m}^{2} \leq K \| \phi \|_{m+2}^{2}$ for a constant $K$, 
  \item and \emph{coercive}, i.e.\ $\langle \mathcal{A}_{i} \phi, \phi \rangle \geq \lambda \| \phi \|_{m+2}^{2}$ for a constant $\lambda > 0$,
  \end{enumerate}
for every $i \in \{1, \dots, n\}$ and for all $\phi \in W^{m+2}_{2}$, where $\langle \cdot{},\cdot{} \rangle$ denotes the duality pairing between $W^{m+2}_{2}$ and $W^{m}_{2}$ based on the inner product in $W^{m+1}_{2}$. Then by the separability of $W^{m+2}_{2}$, there exists a countable dense subset $\{e_{j}\}_{j=1}^{\infty}$ such that for fixed $p\geq 1$, $\phi_{p} = \sum_{j=1}^{p} c_{j}e_{j}$ for constants $c_{j}$ where $\phi_{p} \in W^{m+2}_{2}$ is not identically zero. We fix $i$ and for every $\psi \in W^{m}_{2}$ consider $\mathcal{A}$ acting on $\phi_{p}$, that is $\langle \mathcal{A} \phi_{p}, e_{j} \rangle = \langle \psi, e_{j} \rangle $ for all $j \in \{1, \dots, p\}$. Taking linear combinations we obtain $\langle \mathcal{A} \phi_{p}, \phi_{p} \rangle = \langle \psi, \phi_{p} \rangle$ from which we derive $$\lambda \| \phi_{p}\|_{m+2}^{2} \leq \langle \mathcal{A} \phi_{p}, \phi_{p} \rangle = \langle \psi, \phi_{p} \rangle \leq \|\psi\|_{m} \|\phi_{p}\|_{m+2}$$ by the coercivity of $\mathcal{A}$ and an application of the Cauchy-Schwarz inequality. Thus $\| \phi_{p}\|_{m+2} \leq \frac{1}{\lambda} \| \psi \|_{m}$ and hence (by the reflexivity of $W^{m+2}_{2}$), there exists a subsequence $p_{k}$ such that $\phi_{p_{k}}$ converges weakly to $\phi$ and in particular $\langle \mathcal{A}_{i} \phi_{p_{k}}, e_{j} \rangle \to \langle \mathcal{A} \phi, e_{j} \rangle $ for every $j$. Therefore for every $\psi \in W^{m}_{2}$ there exists a $\phi \in W^{m+2}_{2}$ satisfying $\mathcal{A}_{i} \phi = \psi$ for every $i \in \{1, \dots, n\}$. Moreover, this solution is easily seen to be unique.
%

Using the existence and uniqueness to the elliptic problem in each interval, we note that $$v_{0} + \tau f_{1} + \sum_{\rho=1}^{d_{1}} (\mathcal{M}^{\rho}_{0} v_{0} + g^{\rho}_{0})\xi^{\rho}_{1} \in W^{m}_{2}$$ by Assumption \ref{asm: initial conditions and free terms} and therefore there exists a $v_{1} \in W^{m+2}_{2}$ satisfying $$(I-\tau \mathcal{L}_{1})v_{1} = v_{0} + \tau f_{1} + \sum_{\rho=0}^{d_{1}}(\mathcal{M}^{\rho}_{0} v_{0} + g^{\rho}_{0})\xi^{\rho}_{1}.$$ Further, assuming that there exists a $v_{i} \in W^{m+2}_{2}$ satisfying \eqref{eqn: implicit time scheme} we have that $$v_{i} + \tau f_{i+1} + \sum_{\rho=0}^{d_{1}} (\mathcal{M}^{\rho}_{i}v_{i} + g^{\rho}_{i})\xi^{\rho}_{i+1} \in W^{m}_{2}$$ by the induction hypothesis and Assumption \ref{asm: initial conditions and free terms}, and therefore there exists a $v_{i+1} \in W^{m+2}_{2}$ satisfying \eqref{eqn: implicit time scheme}. Hence we obtain $v=(v_{i})_{i=1}^{n}$ such that each $v_{i} \in W^{m+2}_{2}$ satisfies \eqref{eqn: implicit time scheme}.

It only remains to prove the claim concerning ellipticity of $\mathcal{A}$. By Assumption \ref{asm: boundedness of coefficients}, clearly $\mathcal{A}_{i}$ is a bounded linear operator for each $i$. We see that 
\begin{align*}
\langle \mathcal{A} \phi, \phi \rangle &= \langle I \phi, \phi \rangle - \tau \langle \mathcal{L} \phi, \phi \rangle \\ 
	&= \| \phi\|_{m+1}^{2} - \tau \langle \mathcal{L} \phi , \phi \rangle
\end{align*} 
where, by Assumptions \ref{asm: boundedness of coefficients} and \ref{asm: strong stochastic parabolicity},
\begin{align*} 
\langle \mathcal{L}\phi , \phi \rangle &= ((a^{0\beta} - D_{\alpha}a^{\alpha\beta})D_{\beta} \phi , \phi) - (a^{\alpha\beta} D_{\beta} \phi, D_{\alpha} \phi) \\
	&\leq C \| \phi \|_{m+1}^{2} - \frac{\kappa}{2} \|\phi\|_{m+2}^{2}
\end{align*}
in the $W^{m+1}_{2}$ inner product for $\alpha, \beta \in \{1, \dots, d\}$ and for a constant $C$ depending on $K_{0}$ and $K_{1}$. Therefore $$\langle \mathcal{A}\phi, \phi \rangle \geq \frac{\kappa}{2} \tau \| \phi\|_{m+2}^{2} + (1- \tau K) \| \phi\|_{m+1}^{2} \geq \frac{\kappa}{2} \tau \| \phi\|_{m+2}^{2}$$ for sufficiently small $\tau$ and hence (ii) is satisfied.  

To prove the estimate, we use a method similar to that in the proof of Theorem \ref{thm: Sobolev valued solution to space-time scheme and estimate} to arrive at \eqref{eqn: intermediate equation with H, I, J} with $v$ in place of $\nu$, $\mathcal{L}$ in place of $L^{h}$, $\mathcal{M}^{\rho}$ in place of $M^{h,\rho}$, all in the $W^{m+1}_{2}$-norm instead of the $L^{2}$-norm. Again, we decompose $\mathcal{J}$ into $\mathcal{J}^{(1)}$ and $\mathcal{J}^{(2)}$ using the processes $Y^{\pi \rho}(t)$ that arise by applying the It{\^o} formula to the product of increments of the independent Wiener processes. However this time, instead of using Lemma \ref{lem: bound on Q}, we observe that
$$\mathcal{H}_{j} + \mathcal{J}^{(1)}_{j} \leq N \tau \sum_{i=1}^{j} \| v_{i} \|_{m+1}^{2} -\frac{\kappa}{2} \tau \sum_{i=1}^{j} \| v_{i} \|_{m+2}^{2} + N \tau \sum_{i=0}^{j} (\|f_{i}\|_{m}^{2} + \|g_{i}\|_{m+1}^{2})$$
since
\begin{equation*}
\int_{\mathbf{R}^{d}} 2v(x)\mathcal{L} v(x) + \sum_{\rho=1}^{d_{1}} |\mathcal{M}^{\rho}v(x)|^{2} \, dx \leq (\varepsilon - \kappa C) \| v\|_{m+2}^{2} + C \| v \|_{m+1}^{2}
\end{equation*} 
for $\varepsilon > 0$ by the considerations above. Therefore we have that
\begin{equation*}
\begin{split}
\| v_{j}\|_{m+1}^{2} + \tau \sum_{i=1}^{j}\|v_{i}\|_{m+2}^{2} \leq N \| v_{0} \|_{m+1}^{2} + N \mathcal{I}_{j} + N \mathcal{J}^{(2)}_{j} \\ +  N\tau \sum_{i=0}^{j} (\|f_{i}\|_{m}^{2} + \| g_{i}\|_{m+1}^{2})
\end{split}
\end{equation*}
and the estimate follows by considering the maximum and then taking the expectation. Moreover, with the estimate, it is clear that the solution $v \in \mathbf{W}^{m+2}_{2}(\tau)$.
\end{proof}

We will use the theorem above to obtain estimates in appropriate Sobolev spaces for a system of time discretized equations. For $i \in \{0, \dots, n\}$ and an integer $p \geq 1$, let $$\mathcal{L}^{(0)}_{i} := \sum_{\lambda,\mu \in \Lambda} \mathfrak{a}^{\lambda\mu}_{i} \partial_{\lambda}\partial_{\mu},$$ $$\mathcal{M}^{(0)\rho}_{i} := \sum_{\lambda \in \Lambda} \mathfrak{b}^{\lambda\rho}_{i} \partial_{\lambda},$$
and let 
\begin{equation*}
\begin{split}
\mathcal{L}^{(p)}_{i} := p! \sum_{\lambda,\mu \in \Lambda_{0}} \mathfrak{a}^{\lambda\mu}_{i} \sum_{j=0}^{p} A_{p,j} \partial^{j+1}_{\lambda} \partial^{p-j+1}_{\mu} + (p+1)^{-1} \sum_{\lambda \in \Lambda_{0}} \mathfrak{a}^{\lambda 0}_{i} \partial^{p+1}_{\lambda} \\+ (p+1)^{-1} \sum_{\mu \in \Lambda_{0}} \mathfrak{a}^{0\mu}_{i} \partial^{p+1}_{\mu},
\end{split}
\end{equation*}
\begin{equation*}
\mathcal{M}^{(p)\rho}_{i} := (p+1)^{-1} \sum_{\lambda \in \Lambda_{0}} \mathfrak{b}^{\lambda\rho}_{i} \partial^{p+1}_{\lambda},
\end{equation*}
\begin{equation*}
\mathcal{O}^{h(p)}_{i} := L^{h}_{i} - \sum_{j=0}^{p} \frac{h^{j}}{j!} \mathcal{L}^{(j)}_{i},
\end{equation*}
and 
\begin{equation*}
\mathcal{R}^{h(p)\rho}_{i}:= M^{h,\rho}_{i} - \sum_{j=0}^{p} \frac{h^{j}}{j!} \mathcal{M}^{(j)\rho}_{i}
\end{equation*}
where $A_{p,j}$ is defined by 
\begin{equation}
\label{eqn: Taylor expansion error coefficient}
A_{p,j} = \frac{(-1)^{p-j}}{(j+1)!(p-j+1)!}.
\end{equation} 
For $p \geq 1$, the values of $\mathcal{L}^{(p)} \phi$ and $\mathcal{M}^{(p)\rho} \phi$ are obtain by formally taking the $p$th derivatives in $h$ of $L^{h} \phi$ and $M^{h,\rho} \phi$ at $h = 0$. 

For a positive integer $k \leq m$, the sequences of random fields $v^{(1)}$, \dots, $v^{(k)}$ needed in \eqref{eqn: expansion} will be the embeddings of random variables taking values in certain Sobolev spaces obtained as solutions to a system of time discretized SPDE. Namely, as the solutions to 
\begin{equation}
\label{eqn: system of time discretized spde}
\begin{split}
\nu^{(p)}_{i} = \nu^{(p)}_{i-1} + (\mathcal{L}_{i} \nu^{(p)}_{i} + \sum_{l=1}^{p} C^{l}_{p} \mathcal{L}^{(l)}_{i} \nu^{(p-l)}_{i} )\tau \\ 
	+ (\mathcal{M}^{\rho}_{i-1}\nu^{(p)}_{i-1} + \sum_{l=1}^{p} C^{l}_{p} \mathcal{M}^{(l)\rho}_{i-1} \nu^{(p-l)}_{i-1} )\xi^{\rho}_{i},
\end{split}
\end{equation}
for $p \in \{1,\dots,k\}$ where $C^{l}_{p} = p(p-1)\cdots (p-l+1)/l!$ is the binomial coefficient and $\nu^{(0)}$ is the solution to \eqref{eqn: implicit time scheme} from Theorem \ref{thm: solvability implicit time scheme with estimate}.
\begin{thm}
\label{thm: solvability of the time system with estimate}
Let Assumptions \ref{asm: boundedness of coefficients}, \ref{asm: strong stochastic parabolicity}, \ref{asm: initial conditions and free terms}, and \ref{asm: boundedness of space-time scheme coefficients} hold with $\mathfrak{m} = m \geq k \geq 1$ and let $\nu^{(0)} \in \mathbf{W}^{m+2}_{2}(\tau)$ be the solution to \eqref{eqn: implicit time scheme} with initial condition $u_{0}$ from Theorem \ref{thm: solvability implicit time scheme with estimate}. Then the system \eqref{eqn: system of time discretized spde} with initial condition $$\nu^{(1)}_{0}=\nu^{(2)}_{0} = \dots = \nu^{(k)}_{0} = 0$$ has a unique set of solutions $(\nu^{(p)} )_{p=1}^{k}$ such that each $\nu^{(p)} \in \mathbf{W}^{m+2-p}_{2}(\tau)$. Moreover for each $p \in \{1,\dots, k\}$,
\begin{equation}
\label{eqn: time system estimate}
\begin{split}
E \max_{i \leq n} \| \nu^{(p)}_{i}\|_{m+1-p}^{2} + E \tau \sum_{i=1}^{n} \| \nu^{(p)}_{i} \|_{m+2 -p}^{2} \\ \leq N E \tau \sum_{i = 0}^{n} ( \| f_{i} \|_{m}^{2} + \| g_{i} \|_{m+1}^{2} )
\end{split}
\end{equation}
holds for a constant $N$ depending only on $d$, $d_{1}$, $\Lambda$, $m$, $K_{0}$, \dots, $K_{m+1}$, $A_{0}$, \dots, $A_{m}$, $\kappa$, and $T$. 
\end{thm}

\begin{proof}
For convenience let
\begin{equation*}
F^{(p)}_{i} := \sum_{j=1}^{p} C^{p}_{j} \mathcal{L}^{(j)}_{i} \nu^{(p-j)}_{i}
\end{equation*}
and 
\begin{equation*}
G^{(p) \rho}_{i} := \sum_{j=1}^{p} C^{p}_{j} \mathcal{M}^{(j)\rho}_{i} \nu^{(p-j)}_{i}
\end{equation*}
where we write $G^{(p)} = \sum_{\rho=1}^{d_{1}} G^{(p) \rho}$.

Observe that for each $p \in \{1, \dots, k\}$ the equation for $\nu^{(p)}$ in \eqref{eqn: system of time discretized spde} depends only on $\nu^{(l)}$ for $l \leq p$ and does not involve any of the unknown processes $\nu^{(l)}$ with indices $l > p$. Therefore we shall prove the solvability of the system and the desired properties on $\nu^{(p)}$ recursively using Theorem \ref{thm: solvability implicit time scheme with estimate}.

For $p=1$, we have
\begin{equation}
\label{eqn: system of time discretized spde with p=1}
\nu^{(1)}_{i} = \nu^{(1)}_{i-1} + (\mathcal{L}_{i} \nu^{(1)}_{i} + F^{(1)}_{i})\tau
	+ \sum_{\rho=1}^{d_{1}}(\mathcal{M}^{\rho}_{i-1} \nu^{(1)}_{i-1} + G^{(1)\rho}_{i-1})\xi^{\rho}_{i}.
\end{equation}
Since $\nu^{(0)} \in \mathbf{W}^{m+2}_{2}(\tau)$, for $\mathfrak{m}=m$ we have that $F^{(1)} \in \mathbf{W}^{m-1}_{2}(\tau)$ and $G^{(1)} \in \mathbf{W}^{m}_{2}(\tau)$ by Assumption \ref{asm: boundedness of space-time scheme coefficients}. Hence by Theorem \ref{thm: solvability implicit time scheme with estimate}, there exists a unique $\nu^{(1)} \in \mathbf{W}^{m+1}_{2}(\tau)$ satisfying \eqref{eqn: system of time discretized spde with p=1} with initial condition $\nu^{(1)}_{0} = 0$. Further, $\nu^{(1)}$ is estimated by \eqref{eqn: time system estimate} and thus Theorem \ref{thm: solvability of the time system with estimate} holds with $p=1$.

Now we assume that for $m \geq k \geq 2$ and $p \in \{2,\dots,k\}$ we have unique $\nu^{(1)}, \dots, \nu^{(p-1)}$ solving \eqref{eqn: system of time discretized spde} for $\nu^{(1)} = \dots \nu^{(p-1)} = 0$ with the desired properties. In particular, observe that for $j \in \{1,\dots, p\}$
\begin{equation}
\label{eqn: estimate for script-L-(j) of v-tau(p-j)}
\begin{split}
[\![ \mathcal{L}^{(j)} \nu^{(p-j)} ]\!]_{m-p} 
\leq N [\![ \nu^{(p-j)}]\!]_{m+2-(p-j)}
\end{split}
\end{equation}
and for each $\rho \in \{ 1, \dots, d_{1}\}$
\begin{equation}
\label{eqn: estimate for script-M-rho(j) of v-tau(p-j)}
\begin{split}
[\![ \mathcal{M}^{(i)\rho} \nu^{ (p-j)}]\!]_{m-p+1} 
\leq N [\![ \nu^{ (p-j)}]\!]_{m+1-(p-j)}
\end{split}
\end{equation}
for a constant $N$. Therefore it follows that $F^{(p)}\in \mathbf{W}^{m-p}_{2}(\tau)$ and $G^{(p)} \in \mathbf{W}^{m-p+1}_{2}(\tau)$. Applying Theorem \ref{thm: solvability implicit time scheme with estimate} yields the existence of a unique solution $\nu^{(p)} \in \mathbf{W}^{m-p+2}_{2}(\tau)$ that satisfies \eqref{eqn: system of time discretized spde} with initial condition $\nu^{(p)}_{0} = 0$. Together with \eqref{eqn: estimate for script-L-(j) of v-tau(p-j)} and \eqref{eqn: estimate for script-M-rho(j) of v-tau(p-j)} the estimate from Theorem \ref{thm: solvability implicit time scheme with estimate} implies that \eqref{eqn: time system estimate} holds. Further, the uniqueness of each $\nu^{(p)}$ follows from Theorem \ref{thm: solvability implicit time scheme with estimate}.
\end{proof}

For the convenience of the reader we record the following lemma and two remarks from \cite{GyongyKrylov:2010} that will be used in proving the error estimates.

\begin{lem}
\label{lem: h derivatives of differences}
Let $\phi \in W_{2}^{p+1}$ and $\psi \in W_{2}^{p+2}$ for a nonnegative integer $p$ and let $\lambda,\mu \in \Lambda_{0}$. Set 
$$ \partial_{\lambda} \phi = \lambda^{j} D_{j} \phi \quad \text{ and } \quad \partial_{\lambda\mu} = \partial_{\lambda}\partial_{\mu}.$$ Then we have
\begin{equation}
\label{eqn: pth derivative in h of a single difference}
\frac{\partial^{p}}{(\partial h)^{p}} \delta_{h,\lambda}\phi (x) = \int_{0}^{1} \theta^{p}\partial_{\lambda}^{p+1} \phi(x+h\theta\lambda)\, d\theta
\end{equation}
and 
\begin{equation}
\label{eqn: pth derivative in h of a double difference}
\begin{split}
\frac{\partial^{p}}{(\partial h)^{p}} & \delta_{h,\lambda}\delta_{-h,\mu}\psi (x) \\& = \int_{0}^{1} \!\!\! \int_{0}^{1} (\theta_{1}\partial_{\lambda} - \theta_{2}\partial_{\mu})^{p} \partial_{\lambda\mu} \psi(x + h(\theta_{1}\lambda - \theta_{2}\mu))\, d\theta_{1}d\theta_{2}
\end{split}
\end{equation}
for almost all $x \in \mathbf{R}^{d}$ for each $h \in \mathbf{R}$. Furthermore, for integer $l \geq 0$ if $\phi \in W_{2}^{p+2+l}$ and $\psi \in W_{2}^{p+3+l}$ then
\begin{equation}
\label{eqn: Taylor expansion error for single difference}
\left\| \delta_{h,\lambda}\phi - \sum_{j=0}^{p}\frac{h^{j}}{(j+1)!} \partial_{\lambda}^{j+1} \phi \right\|_{l} \leq \frac{|h|^{p+1}}{(p+2)!} \left\| \partial_{\lambda}^{p+2} \phi \right\|_{l}
\end{equation}
and
\begin{equation}
\label{eqn: Taylor expansion error for double difference}
\begin{split}
\left\| \delta_{h,\lambda} \delta_{-h, \mu} \psi - \sum_{i=0}^{p} h^{i} \sum_{j=0}^{i} A_{i,j} \partial_{\lambda}^{j+1} \partial_{\mu}^{i-j+1} \psi \right\|_{l} \leq N |h|^{p+1} \|\psi \|_{l+p+3},
\end{split}
\end{equation}
where $A_{i,j}$ is defined by \eqref{eqn: Taylor expansion error coefficient} and $N$ depends on $\lambda$, $\mu$, $d$, and $p$.
\end{lem}

\begin{proof} 
It suffices to prove the lemma for $\phi, \psi \in C_{0}^{\infty}(\mathbf{R}^{d})$. For $p=0$, formula \eqref{eqn: pth derivative in h of a single difference} is obtained by applying the Newton-Leibniz formula 
to $\phi(x+\theta h \lambda)$ as a function of $\theta \in [0,1]$. 
Namely,
\begin{equation*}
\phi(x + h \lambda) - \phi(x ) = \int_{u=x}^{u = x+h\lambda} D_{j} \phi(u) \, du = h \int_{\theta=0}^{\theta=1} \lambda^{j}D_{j} \phi(x+ \theta h \lambda)\, d\theta 
\end{equation*}
and therefore $\delta_{h, \lambda} \phi(x) = \int_{0}^{1} \partial_{\lambda} \phi(x+ \theta h \lambda)\, d\theta$. Applying the Newton-Leibniz formula again yields \eqref{eqn: pth derivative in h of a double difference} with $p=0$. After that, for $p \geq 1$ one obtains \eqref{eqn: pth derivative in h of a single difference} and \eqref{eqn: pth derivative in h of a double difference} by differentiating both parts of these equations written with $p=1$.

Next by Taylor's formula for smooth $f(h)$ we have
\begin{equation*}
f(h) = \sum_{j=0}^{p} \frac{h^{j}}{j!} \frac{d^{j}}{(dh)^{j}} f(0) + \frac{1}{p!} \int_{0}^{h} (h-\theta)^{p} \frac{d^{p+1}}{(dh)^{p+1}} f(\theta) \, d\theta.
\end{equation*}
Applying this to 
\begin{equation*}
\delta_{h,\lambda}\phi(x) = \int_{0}^{1} \partial_{\lambda} \phi(x+ \theta h \lambda) \, d\theta
\end{equation*}
as a function of $h$ we see that 
\begin{equation*}
\begin{split}
\delta_{h,\lambda} \phi(x) = & \sum_{j=0}^{p} \frac{h^{j}}{(j+1)!} \partial_{\lambda}^{j+1} \phi(x) \\ &\qquad + \frac{h^{p+1}}{p!} \int_{0}^{1}\!\!\! \int_{0}^{1}(1-\theta_{2})^{p}\theta_{1}^{p+1} \partial_{\lambda}^{p+2} \phi(x+h\theta_{1}\theta_{2}\lambda)\, d\theta_{1}d\theta_{2}.
\end{split}
\end{equation*}
Now to prove \eqref{eqn: Taylor expansion error for single difference}, it remains only to use that by Minkowski's integral inequality the $W_{2}^{l}$-norm of the last term is less than the $W_{2}^{l}$-norm of $\partial_{\lambda}^{p+2} \phi$ times
\begin{equation*}
\frac{|h|^{p+1}}{p!} \int_{0}^{1} \!\!\! \int_{0}^{1} (1-\theta_{2})^{p} \theta_{1}^{p+1} \, d\theta_{1}d\theta_{2} = \frac{|h|^{p+1}}{(p+2)!}.
\end{equation*}
Similarly, by observing that the value at $h=0$ of the right hand side of \eqref{eqn: pth derivative in h of a double difference} is 
\begin{equation*}
p! \sum_{j=0}^{p} A_{p,j} \partial_{\lambda}^{j+1} \partial_{\mu}^{p-j+1} \psi(x),
\end{equation*}
we see that the left hand side of \eqref{eqn: Taylor expansion error for double difference} is the $W_{2}^{l}$-norm of 
\begin{equation*}
\frac{h^{p+1}}{p!} \int_{0}^{1} \!\!\! \int_{0}^{1} \!\!\! \int_{0}^{1} (1-\theta_{3})^{p}(\theta_{1}\partial_{\lambda} -\theta_{2}\partial_{\mu})^{p+1} \partial_{\lambda\mu} \psi (x+ h\theta_{3}(\theta_{1}\lambda -\theta_{2}\mu))\, d\theta_{1}d\theta_{2}d\theta_{3},
\end{equation*}
which yields \eqref{eqn: Taylor expansion error for double difference}.
\end{proof}

For integers $ l \geq 0$ and $r \geq 1$, denote by $W_{h,2}^{l,r}$ the Hilbert space of functions $\phi$ on $\mathbf{R}^{d}$ such that
\begin{equation}
\label{eqn: discrete Sobolev norm}
\| \phi \|_{l,r,h}^{2} := \sum_{\lambda_{1}, \dots, \lambda_{r} \in \Lambda} \| \delta_{h,\lambda_{1}} \times \dots \times \delta_{h, \lambda_{r}} \phi \|_{l}^{2} < \infty
\end{equation}
and set $W_{h,2}^{l,0} = W_{2}^{l}$. Then for any $\phi \in W_{2}^{l+r}$ we have 
\begin{equation*}
\| \phi \|_{l,r,h} \leq N \| \phi \|_{l+r},
\end{equation*}
where $N$ depends only on $|\Lambda_{0}|^{2} := \sum_{\lambda \in \Lambda_{0}} |\lambda|^{2}$ and $r$.

\begin{rmk}
\label{rmk: differences are bounded by derivatives}
Formula \eqref{eqn: pth derivative in h of a single difference} with $p=0$ and Minkowski's integral inequality imply that 
\begin{equation*}
\| \delta_{h,\lambda} \phi \|_{0} \leq \|\partial_{\lambda}\phi \|_{0}.
\end{equation*}
By applying this inequality to finite differences of $\phi$ and using induction we can conclude that $W_{2}^{l+r} \subset W_{h,2}^{l,r}$. 
\end{rmk}

\begin{rmk}
\label{rmk: bound on script-O and script-R}
Owing to Assumption \ref{asm: consistency}, for $i \in \{0, \dots, n\}$ we have that $\mathcal{L}^{(0)}_{i} = \mathcal{L}_{i}$ and $\mathcal{M}^{(0)\rho}_{i} = \mathcal{M}^{\rho}_{i}$. Also by Lemma \ref{lem: h derivatives of differences} and Assumptions \ref{asm: boundedness of space-time scheme coefficients} and \ref{asm: stochastic parabolicity for space-time scheme}, for $\phi \in W^{p+2+l}_{2}$ and $\psi \in W^{p+3+l}_{2}$ we have 
\begin{equation*}
\| \mathcal{O}^{h(p)} \psi \|_{l} \leq N |h|^{p+1} \| \psi\|_{l+p+3}
\end{equation*}
and 
\begin{equation*}
\| \mathcal{R}^{h(p)\rho} \phi\|_{l} \leq N |h|^{p+1} \| \phi \|_{l+p+2}
\end{equation*}
for a constant $N$ depending only on $p$, $d$, $l$, $A_{0}$, \dots, $A_{l}$, and $\Lambda$. 
\end{rmk}

For integers $k, l \geq 0$, let $\nu^{ (0)}, \nu^{ (1)}, \dots, \nu^{ (k)}$ be the functions from Theorem \ref{thm: solvability of the time system with estimate}. We define
\begin{equation}
\label{eqn: auxiliary equation}
r^{\tau,h}_{i} := \nu^{h}_{i} - \nu^{(0)}_{i} - \sum_{j=1}^{k} \frac{h^{j}}{j!} \nu^{(j)}_{i}
\end{equation}
for $i \in \{1, \dots, n\}$ where $\nu^{h}$ is the unique $L^{2}$-valued solution to \eqref{eqn: implicit space-time scheme} that exists by Theorem \ref{thm: Sobolev valued solution to space-time scheme and estimate} with initial condition $u_{0}$, data $f^{0} = f$ and $f^{\mu} = 0$, $\mu \in \Lambda_{0}$.

\begin{lem}
\label{lem: solvability of auxiliary equation}
Let Assumptions \ref{asm: boundedness of coefficients}, \ref{asm: strong stochastic parabolicity}, \ref{asm: initial conditions and free terms}, and \ref{asm: boundedness of space-time scheme coefficients} hold with $\mathfrak{m}=m = l + k + 1$ for integers $k, l \geq 0$ and let $r^{\tau, h}$ be defined as in equation \eqref{eqn: auxiliary equation}. Then $r^{\tau, h}_{0} = 0$, $r^{\tau,h} \in \mathbf{W}^{m-k}_{2}(\tau)$ and 
\begin{equation*}
r^{\tau, h}_{i} = r^{\tau, h}_{i-1} + (L^{h}_{i} r^{\tau, h}_{i} + F^{\tau,h}_{i}) \tau + \sum_{\rho=1}^{d_{1}} (M^{h,\rho}_{i-1} r^{\tau,h}_{i-1} + G^{\tau,h,\rho}_{i-1})\xi^{\rho}_{i}
\end{equation*}
for $i \in \{1, \dots, n\}$ where
\begin{equation*}
F^{\tau,h}_{i} := \sum_{j=0}^{k} \frac{h^{j}}{j!} \mathcal{O}^{h(k-j)}_{i} \nu^{(j)}_{i}
\end{equation*}
and
\begin{equation*}
G^{\tau,h,\rho}_{i-1} := \sum_{j=0}^{k} \frac{h^{j}}{j!} \mathcal{R}^{h(k-j)\rho}_{i-1} \nu^{(j)}_{i-1}
\end{equation*}
and, moreover, $F^{\tau,h} \in \mathbf{W}^{l}_{2}(\tau)$ and $G^{\tau,h,\rho} \in \mathbf{W}^{l+1}_{2}(\tau)$.
\end{lem}

\begin{proof}
By Theorem \ref{thm: Sobolev valued solution to space-time scheme and estimate} the solution to the space-time scheme $\nu^{h} \in \mathbf{W}^{m}_{2}(\tau)$ and by Theorem \ref{thm: solvability implicit time scheme with estimate} the solution to the time scheme $\nu^{(0)} \in \mathbf{W}^{m+2}_{2}(\tau)$. Therefore $r^{\tau,h} \in \mathbf{W}^{m}_{2}(\tau)$ when $k=0$ and, by Theorem \ref{thm: solvability of the time system with estimate}, $r^{\tau,h} \in \mathbf{W}^{m-k}_{2}(\tau)$ when $k \geq 1$.

Observe that 
\begin{align*}
\sum_{i=0}^{k} \frac{h^{i}}{i!} \sum_{j=0}^{k-i} \frac{h^{j}}{j!} \mathcal{L}^{(j)} \nu^{ (i)} &= \sum_{i=0}^{k-1} \frac{h^{i}}{i!} \sum_{j=1}^{k-i} \frac{h^{j}}{j!} \mathcal{L}^{(j)} \nu^{ (i)}\\
	&= \sum_{i=1}^{k} \sum_{j=0}^{k-i} \frac{h^{i+j}}{i!j!} \mathcal{L}^{(i)} \nu^{ (j)}\\
	&= \sum_{i=1}^{k} \sum_{j=i}^{k} \frac{h^{j}}{i!(j-i)!} \mathcal{L}^{(i)} \nu^{ (j- i)}\\
	&= \sum_{i =1}^{k} \sum_{j=1}^{i} \frac{h^{i}}{j!(i-j)!} \mathcal{L}^{(j)} \nu^{ (i-j)} =: I^{\tau,h}
\end{align*}
where summations over empty sets are zero. Therefore, we can rewrite $F^{\tau,h}$ as 
\begin{equation*}
F^{\tau,h} = L^{h} \nu^{ (0)} - \mathcal{L}\nu^{ (0)} + \sum_{j=1}^{k} \frac{h^{j}}{j!} L^{h} \nu^{ (j)} - \sum_{j=1}^{k} \frac{h^{j}}{j!} \mathcal{L} \nu^{ (j)} - I^{\tau,h}.
\end{equation*}
Similarly, observe that 
\begin{equation*}
\sum_{i=0}^{k} \frac{h^{i}}{i!} \sum_{j=0}^{k-i} \frac{h^{j}}{j!} \mathcal{M}^{(j)\rho} \nu^{ (i)} = \sum_{i=1}^{k}\sum_{j=1}^{i} \frac{h^{i}}{j! (i-j)!} \mathcal{M}^{(j)\rho} \nu^{ (i-j)} =: J^{\tau,h,\rho}
\end{equation*}
and therefore 
\begin{equation*}
G^{\tau,h,\rho} = M^{h,\rho}\nu^{(0)} - \mathcal{M}^{\rho}\nu^{(0)} + \sum_{j=1}^{k} \frac{h^{j}}{j!} M^{h,\rho}\nu^{(j)} - \sum_{j=1}^{k} \frac{h^{j}}{j!}\mathcal{M}^{\rho} \nu^{(j)} - J^{\tau,h,\rho}.
\end{equation*}
Thus, following from Remark \ref{rmk: bound on script-O and script-R} and Theorem \ref{thm: solvability of the time system with estimate}, $F^{\tau,h} \in \mathbf{W}^{l}_{2}(\tau)$ and $G^{\tau,h,\rho} \in \mathbf{W}^{l+1}_{2}(\tau)$.
\end{proof}

With the previous considerations, we are now prepared to prove the main results.

\section{Proof of Main Results}\label{sec: Proof of Main Results}

We prove a slightly more general result which implies Theorem \ref{thm: expansion and estimate for differences of expansion error}. Here we suppose that $\mathfrak{m} = m$.  

\begin{thm}
\label{thm: generalized result for auxiliary equation}
Let Assumptions \ref{asm: boundedness of coefficients}, \ref{asm: strong stochastic parabolicity}, \ref{asm: initial conditions and free terms}, \ref{asm: boundedness of space-time scheme coefficients}, \ref{asm: stochastic parabolicity for space-time scheme},  and \ref{asm: consistency} hold with $m = l + k + 1$ for integers $l, k \geq 0$. Then for $r^{\tau, h}$ as defined in \eqref{eqn: auxiliary equation} we have 
\begin{equation}
\label{eqn: estimate for auxiliary equation}
E \max_{i \leq n} \| r^{\tau, h}_{i} \|_{l}^{2} + E \tau \sum_{i=1}^{n} \sum_{\lambda \in \Lambda} \| \delta_{h,\lambda} r^{\tau,h}_{i} \|_{l}^{2} \leq N |h|^{2(k+1)} \mathcal{K}_{m},
\end{equation}
where $N$ depends only on $d$, $d_{1}$, $\Lambda$, $m$, $K_{0}$, \dots, $K_{m+1}$, $A_{0}$, \dots, $A_{m}$, $\kappa$, and $T$.
\end{thm}

\begin{proof}
Recall that by Lemma \ref{lem: solvability of auxiliary equation} we have that $F^{\tau,h} \in \mathbf{W}^{l}_{2}(\tau)$ and $G^{\tau,h,\rho} \in \mathbf{W}^{l+1}_{2}(\tau)$. Then the left-hand-side of \eqref{eqn: estimate for auxiliary equation} is dominated by
\begin{equation}
\label{eqn: bound on F and G}
N E \tau \sum_{i=1}^{n} ( \| F^{\tau,h}_{i} \|_{l}^{2} + \| G^{\tau,h,\rho}_{i} \|_{l}^{2} )
\end{equation}
due to Lemma \ref{lem: solvability of auxiliary equation} and Theorem \ref{thm: Sobolev valued solution to space-time scheme and estimate}. To estimate \eqref{eqn: bound on F and G} we observe that for $j \leq k$, by Remark \ref{rmk: bound on script-O and script-R} we have that
\begin{equation*}
\| \mathcal{O}^{h(k-j)}_{i} \nu^{(j)}_{i} \|_{l} \leq N |h|^{k-j+1} \| \nu^{(j)}_{i} \|_{l+k-j+3} = N |h|^{k-j+1} \|\nu^{(j)}_{i}\|_{m+2-j},
\end{equation*}
and combining this result with Theorem \ref{thm: solvability of the time system with estimate} yields 
\begin{equation*}
E \tau \sum_{i=1}^{n}  \| F^{\tau,h}_{i} \|_{l}^{2} \leq N |h|^{2(k+1)} \mathcal{K}_{m}.
\end{equation*}
The bound on $G^{\tau,h,\rho}$ can be obtained in a similar fashion, yielding the desired result.
\end{proof}

Now set $R^{\tau,h} := I r^{\tau,h}$ where $I$ is the embedding operator from Lemma \ref{lem: embedding}. We have the following corollary to Theorem \ref{thm: generalized result for auxiliary equation} which implies Theorem \ref{thm: expansion and estimate for differences of expansion error}.

\begin{cor}
\label{cor: estimate for error}
If the assumptions of Theorem \ref{thm: generalized result for auxiliary equation} hold with $l > p + d/2$ for a nonnegative integer $p$ then for $\lambda \in \Lambda^{p}$
\begin{equation*}
E \max_{i \leq n} \sup_{x \in \mathbf{R}^{d}} |\delta_{h,\lambda} R^{\tau,h}_{i} (x)|^{2} \leq N h^{2(k+1)} \mathcal{K}_{m}
\end{equation*}
and 
\begin{equation*}
E \max_{i \leq n} \sum_{x \in G_{h}} |\delta_{h,\lambda} R^{\tau,h}_{i} (x)|^{2} |h|^{d} \leq N h^{2(k+1)} \mathcal{K}_{m}
\end{equation*}
hold for a constant $N$ depending only on $d$, $d_{1}$, $\Lambda$, $m$, $K_{0}$, \dots, $K_{m+1}$, $A_{0}$, \dots, $A_{m}$, $\kappa$, and $T$. 
\end{cor}

\begin{proof}
Using Sobolev's embedding of $W^{l-p}_{2}$ into $C_{b}$ and Remark \ref{rmk: differences are bounded by derivatives}, Theorem \ref{thm: generalized result for auxiliary equation} implies
\begin{align*}
E \max_{i \leq n} \sup_{x \in \mathbf{R}^{d}} |\delta_{h,\lambda} R^{\tau,h}_{i} (x) |^{2} &\leq C E \max_{i \leq n} \| r^{\tau,h}_{i} \|_{l-p,p,h}^{2}\\
	&\leq C^{\prime} E \max_{i \leq n} \| r^{\tau,h}_{i} \|_{l}^{2}\\
	&\leq N h^{2(k+1)} \mathcal{K}_{m}
\end{align*}
where $C$ and $C^{\prime}$ are constants depending only on $m$ and $d$, and $N$ is a constant depending only on $m$, $d$, $d_{1}$, $\kappa$, $\Lambda$, $K_{0}$, \dots, $K_{m+1}$, and $T$. Similarly, by Lemma \ref{lem: embedding} above and Remark \ref{rmk: differences are bounded by derivatives}, 
\begin{align*}
E \max_{i \leq n} \sum_{x \in G_{h}} |\delta_{h,\lambda} R^{\tau,h}_{i} (x) |^{2} |h|^{d} & \leq C E \max_{i \leq n} \| \delta_{h,\lambda} R^{\tau,h}_{i} \|_{l-p}^{2}\\
	&\leq C^{\prime} E \max_{i \leq n} \| r^{\tau,h}_{i} \|_{l}^{2}\\
	&\leq N h^{2(k+1)} \mathcal{K}_{m}.
\end{align*}
\end{proof}

For the $I: W^{l}_{2} \to C_{b}$ from Lemma \ref{lem: embedding}, Theorem \ref{thm: expansion and estimate for differences of expansion error} follows by considering the embeddings $\hat{v}^{h} := I \nu^{h}$, where $\nu^{h}$ is the unique $L^{2}$-valued solution to \eqref{eqn: implicit space-time scheme} with initial condition $u_{0}$, and $v^{(j)} := I \nu^{(j)}$ for $j \in \{ 0, \dots, k\}$, where $\nu^{(0)}$ is the unique $L^{2}$-valued solution to \eqref{eqn: implicit time scheme} with initial condition $u_{0}$ and the processes $\nu^{(1)}$, \dots, $\nu^{(k)}$ are the solutions to the system of time discretized SPDE \eqref{eqn: system of time discretized spde} as given in Theorem \ref{thm: solvability of the time system with estimate}. By Theorem \ref{thm: Sobolev valued solution to space-time scheme and estimate}, $\nu^{h}$ is $\mathcal{F}_{i}$-adapted and $W^{l}_{2}$-valued for all $i \in \{1, \dots, n\}$. For each $j \in \{ 1, \dots, k\}$ the $\nu^{(j)}$ are $W^{p+1+k}$-valued processes by Theorem \ref{thm: solvability of the time system with estimate}. Since $l > d/2$ and $p + 1 - k > d/2$ the processes $\hat{v}^{h}$ and $v^{(j)}$ are well defined and clearly \eqref{eqn: auxiliary equation} implies \eqref{eqn: expansion} with $\hat{v}^{h}$ in place of $v^{h}$. That is, we have the expansion for a continuous version of the $L^{2}$-valued solution.

To see that Theorem \ref{thm: expansion and estimate for differences of expansion error} indeed follows from Corollary \ref{cor: estimate for error} we must show that the restriction of the $L^{2}$-valued solution to the grid $G_{h}$, a set of Lebesgue measure zero,  is indeed equal almost surely to the unique $\ell^{2}(G_{h})$-valued solution that one would naturally obtain from \eqref{eqn: implicit space-time scheme}. That is, we must show that
\begin{equation}
\label{eqn: equality of the restriction} 
\hat{v}^{h}_{i} (x) = v^{h}_{i}(x)
\end{equation} 
almost surely for all $i \in \{1, \dots, n\}$ and for each $x \in G_{h}$ where $v^{h}$ is the unique $\mathcal{F}_{i}$-adapted $\ell_{2}(G_{h})$-valued solution of \eqref{eqn: implicit space-time scheme} from  Theorem \ref{thm: l2 valued solution to the space-time scheme}. Therefore, for a compactly supported nonnegative smooth function $\phi$ on $\mathbf{R}^{d}$ with unit integral and for a fixed $x \in G_{h}$ we define $$\phi_{\varepsilon}(y) := \phi\left( \frac{y-x}{\varepsilon} \right)$$ for $y \in \mathbf{R}^{d}$ and $\varepsilon > 0$. Recall, by Remark \ref{rmk: on Sobolev's embedding}, that we can obtain versions of $u_{0}$, $f$, and $g^{\rho}$ that are continuous in $x$. Since $\hat{v}^{h}$ is a $L^{2}$-valued solution of \eqref{eqn: implicit space-time scheme} for each $\varepsilon$, almost surely
\begin{equation*}
\begin{split}
\int_{\mathbf{R}^{d}} \hat{v}^{h}_{i}(y) \phi_{\varepsilon}(y) \, dy = \int_{\mathbf{R}^{d}} \hat{v}_{i-1}(y) \phi_{\varepsilon}(y) \, dy + \tau \int_{\mathbf{R}^{d}} (L^{h}_{i}\hat{v}^{h}_{i} + f_{i})(y) \phi_{\varepsilon}(y) \, dy\\
	 + \sum_{\rho=1}^{d_{1}}\xi^{\rho}_{i} \int_{\mathbf{R}^{d}} ( M^{h,\rho}_{i-1} \hat{v}^{h}_{i-1} + g^{\rho}_{i-1})(y) \phi_{\varepsilon}(y) \, dy
\end{split}
\end{equation*}
for each $i \in \{1, \dots, n\}$. Letting $\varepsilon \to 0$, we see that both sides converge for all $i \in \{1, \dots, n\}$ and $\omega \in \Omega$. Therefore almost surely 
\begin{equation*}
\hat{v}^{h}_{i}(x) = \hat{v}^{h}_{i-1}(x) + (L^{h}_{i}\hat{v}^{h}_{i}(x) + f_{i}(x) ) \tau + \sum_{\rho=1}^{d_{1}}( M^{h,\rho}_{i-1} \hat{v}^{h}_{i-1}(x) + g^{\rho}_{i-1}(x) )\xi^{\rho}_{i}
\end{equation*}
for all $i \in \{1, \dots, n\}$. Moreover by Lemma \ref{lem: embedding}, the restriction of $\hat{v}^{h}$, the continuous version of $\nu^{h}$, onto $G_{h}$ is an $\ell^{2}(G_{h})$-valued process. Hence \eqref{eqn: equality of the restriction} holds, due to the uniqueness of the $\ell^{2}(G_{h})$-valued $\mathcal{F}_{i}$-adapted solution of \eqref{eqn: implicit space-time scheme} for any $\ell^{2}(G_{h})$-valued $\mathcal{F}_{0}$-measurable initial data. This finishes the proof of Theorem \ref{thm: expansion and estimate for differences of expansion error}.

We end with the following generalization of Theorem \ref{thm: acceleration}.

\begin{thm}
If the assumptions of Theorem \ref{thm: expansion and estimate for differences of expansion error} hold with $p=0$ and $\bar{v}^{h}$ as defined in \eqref{eqn: v-bar} then 
\begin{align*}
&E \max_{i \leq n} \sup_{x \in G_{h}} | \bar{v}^{h}_{i} (x) - v^{(0)}_{i}(x) |^2 \\
&+ E \max_{i \leq n} \sum_{x\in G_{h}} | \bar{v}^{h}_{i} (x) - v^{(0)}_{i}(x) |^2 |h|^{d} \leq N |h|^{2(k+1)} \mathcal{K}_{m}
\end{align*}
for a constant $N$ depending only on $d$, $d_{1}$, $\Lambda$, $m$, $K_{0}$, \dots, $K_{m+1}$, $A_{1}$, \dots, $A_{m}$, $\kappa$, and $T$.
\end{thm}

This follows from Theorem \ref{thm: generalized result for auxiliary equation} and the definition of $\bar{v}^{h}$.

\section{Acknowledgements}\label{sec: Acknowledgements}
The author would like to express gratitude towards his supervisor, Professor Istv{\'a}n Gy{\"o}ngy, for the encouragement and helpful suggestions offered during the preparation of these results which will form part of the author's Ph.D.\ thesis.

\bibliographystyle{amsplain}
\providecommand{\bysame}{\leavevmode\hbox to3em{\hrulefill}\thinspace}
\providecommand{\MR}{\relax\ifhmode\unskip\space\fi MR }
\providecommand{\MRhref}[2]{%
  \href{http://www.ams.org/mathscinet-getitem?mr=#1}{#2}
}
\providecommand{\href}[2]{#2}
   

\end{document}